\documentclass[10pt]{amsart}
\usepackage[utf8]{inputenc}
\usepackage{amsmath,amsthm,amssymb,amsfonts,mathrsfs,bbm,comment,import,xcolor, hyperref,graphicx}

\newcommand{\der}{\partial}

\newcommand{\g}{\mathfrak{g}}
\newcommand{\A}{\mathcal{A}}

\newcommand{\One}{\mathbbm{1}}

\newcommand{\gen}[1]{\ensuremath{\langle #1\rangle}}
\newcommand{\Mat}[1]{\text{Mat}_{#1}(\C)}

\newcommand{\Der}{\operatorname{Der}}

\renewcommand{\Re}{\operatorname{Re}}

\newcommand{\tr}{\operatorname{tr}}

\newcommand{\R}{\mathbb{R}}
\newcommand{\C}{\mathbb{C}}

\theoremstyle{definition}
\newtheorem{definition}{Definition}[section]
\newtheorem{lemma}[definition]{Lemma}
\newtheorem{proposition}[definition]{Proposition}
\newtheorem{theorem}[definition]{Theorem}
\newtheorem{corollary}[definition]{Corollary}
\newtheorem{example}[definition]{Example}

\title{Projective Real Calculi and Levi-Civita connections}
\author{Axel Tiger Norkvist}

\begin{document}

\begin{abstract}
    Based on its central role in the framework of real calculi, the existence of the Levi-Civita connection for real calculi over projective modules is studied, with a special emphasis placed on the simple module $M=\C^N$ over $\A=\Mat{N}$. It is demonstrated that existence of the Levi-Civita connection in this case depends on the Lie algebra $\g$ of hermitian derivations, and necessary and sufficient conditions for the possibility of constructing a real calculus for which there exists a Levi-Civita connection are given in this case. Furthermore, in the general case of real calculi over projective modules, necessary and sufficient conditions for the existence of the Levi-Civita connection are given in terms of explicit projection coefficients.
\end{abstract}

\maketitle

\section{Introduction}\label{sec:intro}
In the rapidly advancing and conceptually rich field of noncommutative geometry, it is crucial to examine the subject matter from various angles to fully comprehend the strengths and weaknesses associated with different perspectives.  In this article, we shall consider the derivation-based approach of real calculi, where a metric and the module on which it acts as a bilinear form are explicitly given, drawing direct inspiration from classical Riemannian geometry.

Inspired by the work in \cite{r:leviCivita}, the framework of real calculi was introduced in \cite{aw:curvature.three.sphere} and \cite{aw:cgb.sphere} as a straightforward derivation-based approach to noncommutative geometry where an $\A$-module $M$ generated from a set of derivations on $\A$ acts as a noncommutative analogue of a vector bundle, on which a metric can be given as an explicit biadditive map on $M$. Since their inception, real calculi have been used to study various geometric aspects of classical noncommutative spaces using a noncommutative analogue of the Levi-Civita connection. For instance, the curvature was calculated for the noncommutative 3-sphere (\cite{aw:curvature.three.sphere}) and the noncommutative cylinder (\cite{al:noncommutative.cylinder}), and in \cite{atn:nc.minimal.embeddings} a minimal embedding of the noncommutative torus into the noncommutative 3-sphere was demonstrated. Moreover, a Gauss-Bonnet type theorem was proven for the noncommutative 4-sphere (\cite{aw:cgb.sphere}). After the introduction of real calculus homomorphisms in \cite{atn:nc.minimal.embeddings}, real calculi were also studied as algebraic structures in \cite{tn:rc.finite.noncomm.spaces} with a particular focus dedicated to projective modules and isomorphisms of real calculi over matrix algebras. 

The work in the aforementioned articles has showcased some of the virtues of real calculi as a way to study noncommutative geometry by highlighting examples where it is straightforward to find a Levi-Civita connection, however (as has been mentioned previously in the articles cited above) existence of the Levi-Civita connection is not guaranteed in general. Considering the fundamental theorem of (pseudo-)Riemannian geometry, questions of existence and uniqueness of the Levi-Civita connection is important in any framework of noncommutative Riemannian geometry, and it has been studied in the past from various different perspectives (for instance, see \cite{dvm:connections.central.bimodules}, \cite{ac:ncgravitysolutions}, \cite{bm:starCompatibleConnections}, \cite{r:leviCivita}, \cite{bgm:LC.connections.spectral.triple} and \cite{ail:LC.connections.q-spheres}).

The current article initiates the study of these questions in earnest for real calculi, focusing on finitely generated projective modules. As a main class of examples we consider the simple projective module $\C^N$ over $\A=\Mat{N}$, and necessary and sufficient conditions are given for when there exists a Levi-Civita connection given a Lie algebra $\g$ of derivations on $\A$. In particular it is shown that there exists no Levi-Civita connection in cases where $\g$ is a semisimple Lie algebra. Beyond this, in the general case a set of equations is derived to determine whether the Levi-Civita connection exists for a given real calculus where the module is projective.

The paper is structured as follows. In Section 2 we recall basic preliminaries of real calculi, and go into detail on the construction of real calculi over $\A=\Mat{N}$. Then, in Section 3 we consider the concept of a Levi-Civita connection, and introduce the notion of a pre-calculus as a collection of given data from which we wish to construct a real calculus. In Section 4 we consider the simple projective module $\C^N$ over $\A=\Mat{N}$, and investigate general pre-calculi of this form. It is worked out in some detail when a real calculus can be constructed from this data such that the Levi-Civita connection exists, and in Section 4.2 the discussion is generalized to $\C^N$ as a simple projective module over a general $^*$-algebra $\A$. Finally, in Section \ref{sec:gen.LC.conds} the existence of the Levi-Civita connection given a real calculus where the module is projective is discussed in some detail, and it is outlined how this question can be answered by considering projections of a free module of rank equal to the dimension of the Lie algebra of derivations.

\section{Preliminaries}\label{sec:prelims}
We begin by recalling the basic definitions and results regarding real calculi (see \cite{aw:curvature.three.sphere}) which make out the framework used throughout this article. Real calculi as a concept is a derivation-based approach to noncommutative geometry, where a module over an algebra is viewed as an analogue of a vector bundle over a manifold. 

If $\A=C^{\infty}(\Sigma)$ is the algebra of smooth functions on a differential manifold $\Sigma$, one may consider the set $\Der(\A)$ of derivations on $\A$ as the module $\mathfrak{X}(\Sigma)$ of smooth sections of the tangent bundle $T\Sigma$, and use this to study various geometric aspects of $\Sigma$. 

However, when considering a noncommutative algebra $\A$, the set $\Der(\A)$ does not possess the structure of an $\A$-module in general. 
Therefore, in contrast to the classical scenario it is important to note that the relationship between derivations on $\A$ and elements of an $\A$-module cannot be expected to follow a straightforward one-to-one correspondence. Additionally, unlike the classical case, the set $\Der(\A)$ in the context of a noncommutative algebra $\A$ includes nontrivial inner derivations. Thus, when considering derivations on $\A$ as an analog to tangent vector fields on a manifold, it might be advisable to selectively examine a subset of $\Der(\A)$ rather than considering the entire set. With this in mind, we make the following definition (for a more in-depth discussion of the challenges with viewing derivations as a noncommutative analogue of vector fields mentioned above, see \cite{dvm:connections.central.bimodules}).
\begin{definition}[Real calculus]
	Let $\A$ be a unital $^*$-algebra over $\C$, and let $\g_{\pi}$ denote a real Lie algebra together with a faithful representation $\pi:\g\rightarrow\text{Der}(\A)$ such that $\pi(\der)$ is a hermitian derivation for all $\der\in\g$. Moreover, let $M$ be a right $\A$-module and let $\varphi:\g\rightarrow M$ be a $\R$-linear map such that $M$ is generated by $\varphi(\g)$. Then
	$C_{\A}=(\A,\g_{\pi},M,\varphi)$ is called a \textit{real calculus} over $\A$.
\end{definition}
\noindent
In the following we will write $\der(a)$ to denote the action of $\g$ on $\A$ as derivations instead of the more cumbersome $\pi(\der)(a)$ for an element $\der\in\g$ and $a\in\A$. Moreover, if $\g\subseteq \Der(\A)$ and the representation $\pi$ is left unspecified this is to be interpreted as $\pi$ being the identity map.

One should note that the framework of real calculi only considers real Lie algebras, and the concept of hermitian derivations corresponds to the notion of real tangent vector fields. In the context of a complex $^*$-algebra $\A$, a hermitian derivation $\der\in\Der(\A)$ is a derivation such that 
\begin{equation*}
    \der(a)=\der^*(a)=(\der(a^*))^*,\quad a\in\A.
\end{equation*}
Finally, in the context of real calculi all modules are right modules, and this is the case for all modules mentioned in this article.

Using the $\A$-module $M$ it is possible to define metrics as invertible biadditive maps $M\times M\rightarrow\A$ which take the $\A$-module structure of $M$ into account as well as the $^*$-structure of $\A$. Formally, the definition is as follows. 
\begin{definition}
    Let $\A$ be a $^*$-algebra and let $M$ be a right $\A$-module. A \textit{hermitian form} on $M$ is a map $h:M\times M\rightarrow\A$ with the following properties:
    \begin{itemize}
        \item[$(h1)$.] $h(m_1,m_2+m_3)=h(m_1,m_2)+h(m_1,m_3)$
        \item[$(h2)$.] $h(m_1,m_2a)=h(m_1,m_2)a$
        \item[$(h3)$.] $h(m_1,m_2)=h(m_2,m_1)^*$
    \end{itemize}
    Moreover, if the map $\hat{h}:M\rightarrow M^*$ (where $M^*$ denotes the dual of $M$), defined by 
    \begin{equation*}
        \hat{h}(m)(n)=h(m,n),
    \end{equation*}
    is a bijection then $h$ is said to be an \emph{(invertible) metric}.
\end{definition}

It is worth noting condition $(h3)$ in the above definition in particular. In classical Riemannian geometry one would usually require the metric to be fully symmetric (i.e., that $g(x,y)=g(y,x)$ for all $x,y\in M$) but this condition is generally too restrictive in the context of noncommutative algebras. Moreover, as $M$ is a conceptual analogue of a complex vector bundle over a smooth manifold it is reasonable to let a metric $h$ in the context of real calculi be the noncommutative analogue of a hermitian metric, hence condition $(h3)$. However, we can use the map $\varphi$ of a real calculus $(\A,\g_{\pi},M,\varphi)$ to consider elements of the form $\varphi(\der)\in M$ for $\der\in\g$ as representatives in $M$ of real tangent vector fields. This inspires a symmetry condition for the metric that only applies to elements of this form.

\begin{definition}[Real metric calculus]
    Let $C_{\A}=(\A,\g_{\pi},M,\varphi)$ be a real calculus over $\A$ and let $h$ be a metric on $M$. If
    \begin{equation*}
        h(\varphi(\der_1),\varphi(\der_2))^*=h(\varphi(\der_1),\varphi(\der_2))
    \end{equation*}
    for every $\der_1,\der_2\in\g$ then the pair $(C_{\A},h)$ is called a \textit{real metric calculus}.
\end{definition}

Next we consider affine connections. These are among the most basic objects in classical geometry and their noncommutative counterparts play an essential role in the noncommutative geometry of real calculi.
\begin{definition}
    Let $C_{\A}=(\A,\g_{\pi},M,\varphi)$ be a real calculus over $\A$. An \textit{affine connection} on $\g\times M$ is a map $\nabla:\g\times M\rightarrow M$ satisfying
    \begin{enumerate}
        \item $\nabla_{\der}(m+n)=\nabla_{\der}m+\nabla_{\der}n$
        \item $\nabla_{\lambda\der+\der'}m=\lambda\nabla_{\der}m+\nabla_{\der'}m$
        \item $\nabla_{\der}(ma)=(\nabla_{\der}m)a+m\der(a)$
    \end{enumerate}
    for $m,n\in M$, $\der,\der'\in\g$, $a\in\A$ and $\lambda\in\R$.
\end{definition}

\subsection{Real calculi over $\A=\Mat{N}$}\label{sec:matrix.alg.intro}
Matrix algebras are important to study both from a theoretical and a practical perspective, and they provide a rich collection of interesting examples to consider. The investigation of real calculi over $\A=\Mat{N}$ was initiated in \cite{tn:rc.finite.noncomm.spaces}, where the effect that the map $\varphi$ has on the algebraic structure of real calculi over $\A$ was studied in some detail, and the current article aims to explore these structures further from a geometric perspective, with a special focus on Levi-Civita connections. 

We begin by noting that all derivations on $\A$ are inner, and hence they can be represented as the commutator of a matrix, and it is straightforward to verify that a hermitian derivation $\der$ corresponds to the commutator $[D,\cdot]$ of a unique trace-free and antihermitian matrix $D$. Hence, let $D_1,...,D_n$ form a basis of $\g\subseteq \mathfrak{su}(N)$. We let $\pi:\g\rightarrow\Der(\A)$ be the  $\R$-linear map $\pi(D_i)=[D_i,\cdot]$ for $i=1,...,n$.

There are several choices of modules $M$ over $\A$ that can be used when constructing real calculi, and in this article the main focus will be on the module $M=\C^N$. Assuming that $N>1$, this module is projective but not free. Moreover, any nonzero vector $v\in\C^N$ generates $M$ as a module over $\A$, since any nonzero vector in $\C^N$ can be linearly transformed into any other vector in $\C^N$. Thus, if $\varphi:\g\rightarrow \C^N$ is any $\R$-linear map that is not the trivial map,
it is clear that $\varphi(\g)$ generates $M$ as an $\A$-module and hence $C_{\A}=(\A,\g_{\pi},\C^N,\varphi)$ is a real calculus.

Expanding on this,
let $h$ be a metric on $M=\C^N$. As described in \cite{tn:rc.finite.noncomm.spaces}, the general metric $h$ on $\C^N$ is of the form
\begin{equation*}
    h(u,v)=x\cdot u^{\dagger}v,\quad u,v\in\C^N
\end{equation*}
for $x\in\R\setminus\{0\}$. Hence, if $(C_{\A},h)$ is a real metric calculus, then the symmetry condition
\begin{equation*}
    h(\varphi(D_i),\varphi(D_j))=x\cdot \varphi(D_i)^{\dagger}\varphi(D_j)=x\cdot \varphi(D_j)^{\dagger}\varphi(D_i)=h(\varphi(D_i),\varphi(D_j))^{\dagger}
\end{equation*}
for $i,j\in\{1,...,n\}$ implies that there is a nonzero vector $\hat{v}_0\in\C^N$ and real constants $\mu_i$ such that 
\begin{equation}\label{eq:metric.anchor.map.CN}
    \varphi(D_i)=\mu_i \hat{v}_0,\quad i=1,...,n.
\end{equation}

To characterize connections on the module $\C^N$ we make use of the fact that $\nabla_i=\nabla_{D_i}$ is a linear map for all $i$, implying that there is a unique matrix $X_i$ such that $\nabla_i v=v X_i $ for all $v\in\C^N$. Using the Leibniz condition
\begin{equation*}
    (vA)X_i=\nabla_{i}(vA)=(\nabla_{i} v) A+v D_i(A)=(v X_i)A+v[D_i,A]
\end{equation*}
it follows that the matrix $X_i$ satisfies
\begin{equation*}
   vAX_i=vX_iA+v[D_i,A]\Leftrightarrow v[D_i+X_i,A]=0
\end{equation*}
for all $v\in\C^N$ and $A\in\Mat{N}$, and thus we see that $X_i=t_i\One-D_i$, where $t_i\in\C$ and $\One\in\A$ denotes the identity matrix. Explicitly, we have that
\begin{equation}\label{eq:affine.connection.CN}
    \nabla_i v=v(X_i)=t_i v-vD_i.
\end{equation}
We shall return to real calculi of the form $(\Mat{N},\g_{\pi},\C^N,\varphi)$ in Section 4, where necessary and sufficient conditions on $\g$ and $\varphi$ for the existence of a noncommutative analogue of the Levi-Civita connection from classical Riemannian geometry are given.

\section{Metric pre-calculi and anchor maps}
The fundamental theorem of Riemannian geometry states that there is a unique connection that is compatible with the Riemann metric as well as having vanishing torsion, and this connection is called the Levi-Civita connection. In the context of real metric calculi existence of such a connection is not always guaranteed, and as can be seen later in Section \ref{sec:fin.dim.modules} existence of a Levi-Civita connection for a given real metric calculus $\big((\A,\g_{\pi},M,\varphi),h\big)$ may depend not only on the algebraic structure of $\g$, but also on how elements of $\g$ act on $\A$ as derivations.

Before considering the notions of metric compatibility and torsion for affine connections $\nabla$ in the context of real calculi, 
we shall require an additional hermiticity condition for elements of the form $\nabla_{\der}\varphi(\der')$ which can be seen as a noncommutative version of taking the covariant derivative of a real vector field with respect to another real vector field.
\begin{definition}[Real connection calculus]\label{def:rcc}
    Let $(C_{\A},h)$ be a real metric calculus and let $\nabla$ be an affine connection on $\g\times M$. If $\nabla$ satisfies
    \begin{equation*}
        h(\nabla_{\der}\varphi(\der_1),\varphi(\der_2))=h(\nabla_{\der}\varphi(\der_1),\varphi(\der_2))^*
    \end{equation*}
    for every $\der,\der_1,\der_2\in\g$, then $(C_{\A},h,\nabla)$ is a \emph{real connection calculus}.
\end{definition}

Using the map $\varphi$ it is possible to define the torsion of a connection in analogy with the classical case.
\begin{definition}[Torsion of a connection]\label{def:torsion}
Let $C_{\A}$ be a real calculus and let $\nabla$ be an affine connection on $\g\times M$. The torsion $T_{\varphi}:\g\times\g\rightarrow M$ is defined as the $\R$-bilinear map satisfying
\begin{equation*}
    T_{\varphi}(\der_1,\der_2)=\nabla_{\der_1} \varphi(\der_2)-\nabla_{\der_2}\varphi(\der_1)-\varphi([\der_1,\der_2]),\quad \der_1,\der_2\in \g.
\end{equation*}
\end{definition}

\begin{definition}[Levi-Civita connection]
    Let $(C_{\A},h,\nabla)$ be a real connection calculus. We say that $\nabla$ is compatible with $h$ if
    \begin{equation*}
        \der(h(m,n))=h(\nabla_{\der}m,n)+h(m,\nabla_{\der}n)
    \end{equation*}
    for every $\der\in\g$ and $m,n\in M$, and \textit{torsion-free} if
    \begin{equation*}
        T_{\varphi}(\der_1,\der_2)=0,\quad \der_1,\der_2\in\g.
    \end{equation*}
    A metric and torsion-free connection is called a \emph{Levi-Civita} connection.
\end{definition}

As has been previously stated, in the general setup of real connection calculi the existence of a Levi-Civita connection can not be guaranteed. However, it is unique if it exists.

\begin{theorem}[\cite{aw:curvature.three.sphere}]\label{thm:existence.and.uniqueness.of.LC}
Let $(C_{\A},h)$ be a real metric calculus. Then there exists at most one Levi-Civita connection $\nabla$ such that $(C_{\A},h,\nabla)$ is a real connection calculus.$\hfill\square$
\end{theorem}

Given a real connection calculus $(C_{\A},h,\nabla)$ where $\nabla$ is a Levi-Civita connection, the uniqueness result in Theorem \ref{thm:existence.and.uniqueness.of.LC} implies that $\nabla$ can (in principle) be recovered from the real metric calculus $(C_{\A},h)$. Moreover, as shall be demonstrated later in this article, one cannot assume that there exists a Levi-Civita connection for a given real metric calculus $(C_{\A},h)$. Hence, we make the following definition.
\begin{definition}
    Let $(C_{\A},h)$ be a real metric calculus. If there exists a Levi-Civita connection $\nabla$ such that $(C_{\A},h,\nabla)$ is a real connection calculus, then $(C_{\A},h)$ is called a \emph{pseudo-Riemannian calculus}.
\end{definition}

A noncommutative version of the classical Koszul formula can be used to determine whether a given connection $\nabla$ is the Levi-Civita connection satisfying the hermiticity condition in Definition \ref{def:rcc}. In the context of real connection calculi we state it as follows.

\begin{proposition}[\cite{aw:curvature.three.sphere}]\label{prop:koszul.formula}
Let $(C_{\A},h,\nabla)$ be a real connection calculus where $\nabla$ is a Levi-Civita connection, and assume that $\der_1,\der_2,\der_3\in\g$. Then $\nabla$ satisfies the equality
\begin{multline}\label{eqn:Koszul}
    2h(\nabla_{\der_1} e_2,e_3)=\der_1h(e_2,e_3)+\der_2h(e_1,e_3)-\der_3h(e_1,e_2)\\
    -h(e_1,\varphi([\der_2,\der_3]))+h(e_2,\varphi([\der_3,\der_1]))+h(e_3,\varphi([\der_1,\der_2])),
\end{multline}
where $e_i=\varphi(\der_i)$ for $i=1,2,3$.
Conversely, if $(C_{\A},h)$ is a real metric calculus and $\nabla$ is a connection satisfying Koszul's formula (\ref{eqn:Koszul}) for every $\der_1,\der_2,\der_3\in\g$, then $(C_{\A},h)$ is a pseudo-Riemannian calculus and $\nabla$ is the unique Levi-Civita connection such that $(C_{\A},h,\nabla)$ is a real connection calculus.$\hfill\square$
\end{proposition}

Given a unital $^*$-algebra $\A$, a Lie algebra $\g\subseteq \Der(\A)$ and a right $\A$-module $M$ there are in general many maps $\varphi:\g\rightarrow M$ such that $C_{\A}=(\A,\g_{\pi},M,\varphi)$ is a real calculus, and in \cite{tn:rc.finite.noncomm.spaces} the notion of real calculus isomorphisms were used to demonstrate how the choice of $\varphi$ could lead to nonisomorphic real calculi in cases where $\A$, $\g_{\pi}$ and $M$ were fixed. Generally, what constitutes a natural choice of map $\varphi$ is not clear from the definitions and it is interesting to see in what ways the choice of $\varphi$ affects the resulting real calculus. Hence, we make the following definition.
\begin{definition}
    Let $\A$ be a unital $^*$-algebra, let $\g$ be a Lie algebra, let $\pi:\g\rightarrow \Der(\A)$ be a faithful representation of $\g$, and let $M$ be a right module over $\A$. Then the structure $B_{\A}=(\A,\g_{\pi},M)$ is called a \textit{pre-calculus}. Moreover, if $h:M\times M\rightarrow \A$ is a metric, then $(B_{\A},h)$ is called a \textit{metric pre-calculus}.
    
    If $\varphi:\g\rightarrow M$ is a $\R$-linear map such that $\varphi(\g)$ generates $M$ as a module, then $\varphi$ is called an \textit{anchor map}. Moreover, if $h(\varphi(\der),\varphi(\der'))$ is hermitian for every $\der,\der'\in\g$, then $\varphi$ is called a \emph{metric anchor map}.
\end{definition}
Given a metric pre-calculus $(B_{\A},h)$ it is possible to define affine connections $\nabla$ and determine whether they are compatible with the metric $h$ without the need of a metric anchor map $\varphi:\g\rightarrow M$. However, as can be seen from 
Definition \ref{def:torsion} the torsion of a connection directly depends on the choice of $\varphi$. This motivates the study of how the choice of anchor map affects the existence of a Levi-Civita connection given a fixed metric pre-calculus $((\A,\g,M),h)$ and, in particular, whether there exist metric pre-calculi for which no metric anchor map is such that the resulting real metric calculus is pseudo-Riemannian.

\section{Finite-dimensional simple modules}\label{sec:fin.dim.modules}
We shall treat the question of existence of a metric anchor map given a fixed metric pre-calculus $(B_{\A},h)$ such that the resulting real metric calculus is pseudo-Riemannian in the special case where $B_{\A}=(\Mat{N},\g_{\pi},\C^N)$. This scenario was considered in Section \ref{sec:matrix.alg.intro}, and a brief summary of that discussion is given below. The general metric $h$ on $M$ is given by $h(u,v)=x\cdot u^{\dagger}v$,   $x\in\R\setminus\{0\}$, and we can assume that $\g\subseteq\mathfrak{su}(N)$ without loss of generality.
Given a basis $D_1,...,D_n$ of $\g$, we let $\pi(D_i)=[D_i,\cdot]\in\Der(\A)$ for $i=1,...,n$, making $((\Mat{N},\g_{\pi},\C^N),h)$ into a metric pre-calculus. 
A map $\varphi:\g\rightarrow\C^N$ is a metric anchor map if and only if there is a nonzero vector $\hat{v}_0\in\C^N$ and constants $\mu_j\in\R$ such that
\begin{equation*}
    \varphi(D_j)=\mu_j\hat{v}_0,\quad j=1,...,n;
\end{equation*}
since an anchor map cannot be the trivial map, there must be at least one $j\in\{1,...,n\}$ such that $\mu_j\neq 0$.
As a final note before moving forward, connections on $\C^N$ can be parameterized by
\begin{equation*}
    \nabla_{D_j} v=v X_j=-vD_j+t_j v, \quad v\in\C^N,
\end{equation*}
where $X_j=-D_j+t_j\One_N$ and $t_j\in\C$.

Expanding on these preliminaries, we now describe metric compatibility of a connection $\nabla$. Calculating the sum 
$h(\nabla_j u,v)+h(u,\nabla_j v)$ explicitly (using the identity $D_j^{\dagger}=-D_j$), we get
\begin{align*}
   h(\nabla_j u,v)+h(u,\nabla_j v)&=h(t_j u-uD_j,v)+h(u,t_j v-vD_j)\\
   &=x\left[(\bar{t}_j+t_j)u^{\dagger}v+D_j u^{\dagger}v-u^{\dagger}v D_j\right] \\
   &=x\left[(\bar{t}_j+t_j)u^{\dagger}v+[D_j, u^{\dagger}v]\right]\\
   &=x(\bar{t}_j+t_j)u^{\dagger}v+D_j(h(u,v)),
\end{align*}
and hence
\begin{equation*}
    D_j( h(u,v))-\big(h(\nabla_j u,v)+h(u,\nabla_j v)\big)=-x(\bar{t}_j+t_j)u^{\dagger}v
\end{equation*}
is zero for all $u,v\in\C^N$ if and only if $t_j=i\lambda_j$ where $\lambda_j\in\R$ for $j=1,...,n$.

\begin{example}\label{ex:su2}
Below we give an example of a metric pre-calculus where there exists no metric connection $\nabla$ and metric anchor map $\varphi$ such that the torsion $T_{\varphi}$ vanishes everywhere. Let $N=2$ and $\g=\R\gen{D_1,D_2,D_3}=\mathfrak{su}(2)$, where $D_1, D_2$ and $D_3$ are given by
\begin{equation*}
    D_1=\begin{pmatrix}
    0 & i\\
    i & 0
    \end{pmatrix},\quad D_2=\begin{pmatrix}
    0 & 1\\
    -1 & 0
    \end{pmatrix},\quad D_3=\begin{pmatrix}
    i & 0\\
    0 & -i
    \end{pmatrix}.
\end{equation*}
In this basis of $\g$, the Lie bracket is described by the structure constants $f^k_{ij}$ (i.e., $[D_i,D_j]=f^k_{ij}D_k$, where summation over the index $k$ is implied as per the Einstein summation convention), where
\begin{align*}
    &(f^1_{12},f^2_{12},f^3_{12})=(0,0,-2)\\
    &(f^1_{13},f^2_{13},f^3_{13})=(0,2,0)\\
    &(f^1_{23},f^2_{23},f^3_{23})=(-2,0,0).
\end{align*}
As before, the metric $h$ is given by $h(u,v)=x\cdot u^{\dagger}v$, $u,v\in\C^2$ and $x\in\R\setminus\{0\}$, and setting $B_{\A}=(\A,\g_{\pi},\C^2)$ we have that $(B_{\A},h)$ is a metric pre-calculus. We shall show that there is no metric anchor map $\varphi:\g\rightarrow \C^2$ such that the resulting real metric calculus $(C_{\A},h)$ is pseudo-Riemannian by considering the metric compatibility and torsion of a connection $\nabla$ on $\g\times\C^2$.

As before, $\nabla_j v=t_j v-v D_j=v(t_j\One-D_j)=vX_j$ for $t_j\in \C$ and since the metric compatibility of $\nabla$ is equivalent to $t_j=i\lambda_j$ for some $\lambda_j\in\R$ we only consider these choices of $t_i$ going forward. To consider the torsion of $\nabla$ we need an anchor map $\varphi$, and for the sake of readability we write $\varphi_j=\varphi(\der_j)\in\C^2$ and  $h_{ij}=x\varphi_i^{\dagger}\varphi_j$; as noted earlier each $\varphi_j$ must satisfy
\begin{equation*}
    \varphi_j=\mu_j \hat{v}_0,\quad \mu_j\in\R,\quad 0\neq \hat{v}_0\in \C^N
\end{equation*}
if  $\big((\Mat{2},\g_{\pi},\C^N,\varphi),h\big)$ is to be a real metric calculus, and $\mu_j\neq 0$ for at least one $j\in\{1,2,3\}$.

With this notation, we have that
\begin{equation*}
    \nabla_j\varphi_k=\varphi_kX_j=\mu_k \hat{v}_0 X_j=\hat{v}_0(\mu_k X_j),
\end{equation*}
and the torsion becomes
\begin{equation*}
    T_{\varphi}(D_i,D_j)=\nabla_i \varphi_j-\nabla_j\varphi_i-\varphi([D_i,D_j])=\hat{v}_0(\mu_j X_i-\mu_i X_j-f^k_{ij}\mu_k\One).
\end{equation*}
Let $T_{ij}=\mu_j X_i-\mu_i X_j-f^k_{ij}\mu_k\One$, so that $\hat{v}_0 T_{ij}=T_{\varphi}(\der_i,\der_j)$. Then the torsion vanishes if and only if $\hat{v}_0$ is an eigenvector of each $T_{ij}$ with eigenvalue $\lambda_{ij}=0$. Given this, we note that $\mu_j X_i-\mu_i X_j=T_{ij}+f^k_{ij}\mu_k\One$ implies that $\hat{v}_0$ is an eigenvector of $\mu_j X_i-\mu_i X_j$ with eigenvalue $\lambda_{ij}+f^k_{ij}\mu_k=f^k_{ij}\mu_k\in\R$. Moreover, since $\mu_j X_i-\mu_i X_j$ is antihermitian it follows that its spectrum consists of purely imaginary numbers, and hence it follows that $f^k_{ij}\mu_k=0$, $i,j\in\{1,2,3\}$.

Calculating these sums explicitly for $i,j\in\{1,2,3\}$ yields
\begin{equation*}
    f^k_{12}\mu_k=2\mu_3=0,\quad f^k_{13}\mu_k=-2\mu_2=0,\quad f^k_{23}\mu_k=2\mu_1=0.
\end{equation*}
Thus, the torsion cannot vanish unless $\varphi$ is the zero map, which is not an anchor map. Hence, there is no metric anchor map $\varphi$ such that $\big((\Mat{2},\mathfrak{su}(2)_{\pi},\C^N,\varphi),h\big)$ is a pseudo-Riemannian calculus.
\end{example}
\subsection{The general case when $\A=\Mat{N}$}
The above example can be generalized to arbitrary metric pre-calculi $(B_{\A},h)$, where $B_{\A}=(\Mat{N},\g_{\pi},\C^N)$ and $h(u,v)=x\cdot u^{\dagger}v$ for $u,v\in\C^N$ where $x\in\R\setminus\{0\}$. As it turns out, the existence of a metric anchor map $\varphi:g\rightarrow\C^N$ such that the resulting real metric calculus $(C_{\A},h)$ is pseudo-Riemannian generally depends on the Lie algebra $\g\subset\mathfrak{su}(N)$, and thus it is necessary to understand precisely what Lie algebras $\g$ that are possible. 

\begin{proposition}\label{prop:levi.decomposition.suN}
Let $\g\subseteq\mathfrak{su}(N)$ be a nonabelian Lie algebra for $N\geq 2$. Then $\g\simeq \mathfrak{ab}(m)\oplus\g_{ss}$ is a direct sum, where $\g_{ss}$ is semisimple and $\mathfrak{ab}(m)$ denotes the abelian Lie algebra of dimension $m\geq 0$.
\end{proposition}
\begin{proof}
By the Levi decomposition theorem it is clear that if $\g$ is nonabelian it can be decomposed as the semidirect sum of $\operatorname{rad}(\g)$ and a semisimple Lie algebra $\g_{ss}$. Since $\operatorname{rad}(\g)$ is a solvable ideal of $\g$ it is also a solvable Lie subalgebra of $\mathfrak{su}(N)$. Now, since $\mathfrak{su}(N)$ is a compact Lie algebra it follows that $\operatorname{rad}(\g)$ is a compact solvable Lie algebra, and hence abelian, i.e., there exists an $m\geq 0$ such that $\operatorname{rad}(\g)\simeq\mathfrak{ab}(m)$, and hence we may assume without loss of generality that it consists of purely diagonal matrices $D$. Furthermore, since $\operatorname{rad}(\g)$ is an ideal of $\g$, it follows that $[X,D]\in\operatorname{rad}(\g)$ for any $X\in\g_{ss}$ and any $D\in\operatorname{rad}(\g)$. Moreover, by direct computation, it is clear that the diagonal entries of any such matrix $[X,D]$ are all zero, implying that $[X,D]=0$ for $X\in\g_{ss}$ and $D\in\operatorname{rad}(\g)\simeq\mathfrak{ab}(m)$. Hence the Levi decomposition is a direct sum, finishing the proof.
\end{proof}

The result in Example \ref{ex:su2}, where it was not possible to find a metric anchor map such that the torsion $T_{\varphi}$ vanishes, is ultimately due to a simple general fact about semisimple Lie algebras which we state below.

\begin{lemma}\label{lem:g.is.semisimple}
If $\g$ is a semisimple real Lie algebra with basis $\der_1,...,\der_n$ and Lie bracket given by $[\der_i,\der_j]=f^k_{ij}\der_k$, $f^k_{ij}\in\R$ for $i,j,k\in\{1,...,n\}$, then the sum $c_k f^k_{ij}=0$ for all $i,j$ (where each $c_k\in\R$) if and only if $c_k=0$ for all $k$. 
\end{lemma}
\begin{proof}
Let $B$ denote the Killing form on $\g$. Since $\g$ is assumed to be semisimple, Cartan's criterion implies that $B$ is nondegenerate. Given a basis $\der_1,...,\der_n$ for $\g$, let $B_{ij}=B(\der_j,\der_k)$ denote the components of $B$ in this basis and let $B^{ij}$ denote the components of the inverse of $B$, i.e., $B_{ij}B^{jk}=B^{kj}B_{ji}=\delta^k_i$, where $\delta^k_i$ denotes the Kronecker delta. It follows that
\begin{equation*}
    B([\der_i,\der_j],\der_k)=f^l_{ij}B(\der_l,\der_k)=f^l_{ij}B_{lk},
\end{equation*}
and hence it follows that
\begin{equation*}
    f^k_{ij}=B([\der_i,\der_j],\der_l)B^{lk},
\end{equation*}
implying that
\begin{equation*}
    f^k_{ij}c_k=B([\der_i,\der_j],\der_l)B^{lk}c_k=B([\der_i,\der_j],\der_l B^{lk}c_k).
\end{equation*}
Since $[\g,\g]=\g\neq 0$, this expression is zero for all $i,j=1,...,n$ if and only if
\begin{equation*}
    B^{lk}c_k=0
\end{equation*}
for all $l=1,...,n$. Consequently,
\begin{equation*}
    0=B_{il}B^{lk}c_k=\delta_i^k c_k=c_i=0
\end{equation*}
for $i=1,...,n$. The statement follows.
\end{proof}

\begin{proposition}\label{prop:torsion.semisimple.lie.alg}
Let $\big((\Mat{N},\g_{\pi},\C^N,\varphi),h\big)$ be a real metric calculus such that $\g\subseteq\mathfrak{su}(N)$ is a semisimple Lie algebra. Then every connection $\nabla:\g\times \C^N\rightarrow\C^N$ that is compatible with $h$ has non-vanishing torsion $T_{\varphi}$.
\end{proposition}
\begin{proof}
Let $D_1,...,D_n$ be a basis of $\g$ and let $\nabla$ be an affine connection on $\g\times\C^N$ that is compatible with $h$. Then, as noted in the beginning of this section, there are $\lambda_1,...,\lambda_n\in\R$ such that
\begin{equation*}
    \nabla_j v=v X_j=v(i\lambda_j\One-D_j),\quad j=1,...,n.
\end{equation*}
Since $\varphi:\g\rightarrow\C^N$ is a metric anchor map there is a nonzero vector $\hat{v}_0\in\C^N$ and $\mu_j\in\R$ such that $\mu_j\neq 0$ for at least one $j\in\{1,...,n\}$ and such that
\begin{equation*}
    \varphi(D_j)=\mu_j\hat{v}_0,\quad j=1,...,n.
\end{equation*}
Considering the torsion $T_{\varphi}$ we find that
\begin{equation*}
    T_{\varphi}(\der_i,\der_j)=\hat{v}_0(\mu_j X_i-\mu_i X_j-\mu_kf^k_{ij}\One),
\end{equation*}
and we note that the torsion vanishes iff 
\begin{equation*}
    \hat{v}_0(\mu_j X_i-\mu_i X_j)=\hat{v}_0 (\mu_k f^k_{ij}).
\end{equation*}
Assume now that the torsion vanishes everywhere.
in particular we note that, since $\g$ is a real Lie algebra and each $\mu_i\in\R$, the sum $\mu_k f^k_{ij}\in\R$, i.e., that $\hat{v}_0$ is an eigenvector with a real eigenvalue to the matrix $\mu_j X_i-\mu_i X_j$. However, since each $X_j$ is antihermitian it follows that $\mu_j X_i-\mu_i X_j$ is an antihermitian matrix for every $i,j$, implying that all eigenvalues of $\mu_j X_i-\mu_i X_j$ are purely imaginary. Consequently, since the sum $\mu_k f^k_{ij}$ is an eigenvalue to $\mu_j X_i-\mu_i X_j$, it follows that $\mu_k f^k_{ij}=0$. And since $\g$ is semisimple, Lemma \ref{lem:g.is.semisimple} implies that $\mu_k=0$ for all $k$. Hence $\varphi$ is the zero map, which is a contradiction since $\varphi$ was assumed to be a metric anchor map, and the result follows.
\end{proof}

So far we have only considered metric compatibility and torsion of a connection $\nabla:\g\times\C^N\rightarrow\C^N$. However, in the context of real calculi one also has to take into account the condition that $(C_{\A},h,\nabla)$ is a real connection calculus, i.e., that  $h(\nabla_{\der_i}\varphi(\der_j),\varphi(\der_k))$ is hermitian for all $i,j,k$. This condition restricts the possibilities for the choice of anchor map $\varphi$, as the following lemma shows.

\begin{lemma}\label{lem:RCC.cond.matrix}
Let $\big((\Mat{N},\g_{\pi},\C^N,\varphi),h\big)$ be a real metric calculus where $D_1,...,D_n$ is a basis of $\g$ and $\varphi:D_j\mapsto \mu_j \hat{v}_0$, $\mu_j\in\R$, $j=1,...,n$. Given a metric connection $\nabla:\g\times\C^N\rightarrow\C^N$, $(C_{\A},h,\nabla)$ is a real connection calculus iff $\nabla_{D} \hat{v}_0=0$ for $D\in\g$.
\end{lemma}
\begin{proof}
If $\nabla_k\hat{v}_0 =0$ for all $k$, then $h(\nabla_{i}\varphi(D_j),\varphi(D_k))=h(0,\varphi(D_k))=0$ is trivially hermitian for all $i,j,k$, and hence sufficiency of the given condition immediately follows. 

To prove necessity we assume that $h(\nabla_{i}\varphi(\der_j),\varphi(\der_k))$ is hermitian for $i,j,k$ such that $\mu_j\mu_k\neq 0$; that such $j$ and $k$ exist follows from the fact that $\varphi$ is not the zero map. Since $\nabla_{\der_i}$ is a linear map there is, for $i=1,...,n$, a unique matrix $X_i\in \Mat{N}$ such that $\nabla_i \hat{v}_0=\hat{v}_0 X_i$, and by explicit calculation one gets:
\begin{align*}
    0&=h(\nabla_{i}\varphi(D_j),\varphi(D_k))-h(\nabla_{i}\varphi(D_j),\varphi(D_k))^{\dagger}\\
    &=h(\nabla_i\mu_j\hat{v}_0,\mu_k\hat{v}_0)-h(\mu_k\hat{v}_0,\nabla_i\mu_j\hat{v}_0)\\
    &=\mu_j\mu_k X_i^{\dagger}\hat{v}_0^{\dagger}\hat{v}_0-\mu_j\mu_k\hat{v}_0^{\dagger}\hat{v}_0 X_i,
\end{align*}
i.e., it follows that $X_i^{\dagger}\hat{v}_0^{\dagger}\hat{v}_0=\hat{v}_0^{\dagger}\hat{v}_0 X_i$. Assuming, without loss of generality, that $\hat{v}_0\hat{v}_0^{\dagger}=||\hat{v}_0||^2=1$, we get that
\begin{equation*}
    \hat{v}_0 X_i=\hat{v}_0(\hat{v}_0^{\dagger}\hat{v}_0 X_i)=\hat{v}_0(X_i^{\dagger}\hat{v}_0^{\dagger}\hat{v}_0)=(\hat{v}_0X_i^{\dagger}\hat{v}_0^{\dagger})\hat{v}_0,
\end{equation*}
i.e., $\hat{v}_0$ is an eigenvector of $X_i$ with eigenvalue $\lambda_i=\hat{v}_0X_i^{\dagger}\hat{v}_0^{\dagger}$. Moreover, since $X_i^{\dagger}=-X_i$, this implies that
\begin{equation*}
    \lambda_i=\hat{v}_0 X^{\dagger}_i\hat{v}_0^{\dagger}=\hat{v}_0(-X_i)\hat{v}_0^{\dagger}=-\lambda_i\hat{v}_0\hat{v}_0^{\dagger}=-\lambda_i,
\end{equation*}
i.e., that $\lambda_i=0$. Hence, $\nabla_i \hat{v}_0=\hat{v}_0 X_i=0$ and the statement follows.
\end{proof}
Let $\big((\Mat{N},\g_{\pi},\C^N),h\big)$ be a metric pre-calculus, where $h(u,v)=x\cdot u^{\dagger}v$, $x\in\R\setminus\{0\}$. If $\g$ is a semisimple Lie algebra, Proposition \ref{prop:torsion.semisimple.lie.alg} implies that no metric connection $\nabla$ on $\g\times\C^N$ has vanishing torsion with respect to a metric anchor map $\varphi:\g\rightarrow\C^N$. However, if $\g$ is not semisimple and $\varphi:\g\rightarrow\C^N$ is a metric anchor map, then Lemma \ref{lem:RCC.cond.matrix} implies that there is at most one affine connection $\nabla:\g\times\C^N\rightarrow\C^N$ such that $\big((\Mat{N},\g_{\pi},\C^N,\varphi),h,\nabla\big)$ is a real connection calculus, and if this is the case then
\begin{equation*}
    \nabla_D \varphi(D')=0,\quad D,D'\in\g.
\end{equation*}
We give necessary and sufficient conditions for when there exists a metric anchor map $\varphi$ such that the above equation defines a connection $\nabla$ that has vanishing torsion with respect to $\varphi$.
\begin{theorem}\label{thm:riemannian.criterion.matrix}
Let $\big((\Mat{N},\g_{\pi},\C^N),h\big)$ be a metric pre-calculus.
Then there exists a metric anchor map $\varphi:\g\rightarrow\C^N$ such that the resulting real metric calculus $\big((\Mat{N},\g_{\pi},\C^N,\varphi),h\big)$ is pseudo-Riemannian if and only if $\g\subseteq\mathfrak{su}(N)$ is not semisimple and there exists a common eigenvector to all matrices in $\g$.
\end{theorem}
\begin{proof}
We begin by proving sufficiency of the given condition.
If $\g$ is not semisimple, then Proposition \ref{prop:levi.decomposition.suN} implies that $\g$ can be written as a direct sum of an abelian and a semisimple Lie subalgebra, which we denote by $\g^{ab}$ and $\g^{ss}$, respectively. Let $D^{ab}_1,...,D^{ab}_p\in\g^{ab}$ be a basis of $\g^{ab}$ and let $D^{ss}_1,...,D_q^{ss}\in\g^{ss}$ be a basis of $\g^{ss}$, making $D^{ab}_1,...,D_p^{ab},D_1^{ss},...,D^{ss}_q$ a basis of $\g$. 

Let $0\neq\hat{v}_0\in\C^N$ be a common eigenvector of $D^{ab}_1,...D^{ab}_p,D^{ss}_1,...,D^{ss}_q$, with corresponding eigenvalues $i\lambda_1^{ab},...,i\lambda^{ab}_p,i\lambda^{ss}_1,...,i\lambda^{ss}_q$. For $j=1,...,p$, let $\mu^{ab}_j\in\R$ be such that at least one $\mu^{ab}_j$ is nonzero. Now let $\varphi:\g\rightarrow\C^N$ be defined by $\varphi(D^{ab}_j)=\mu^{ab}_j\hat{v}_0$ and $\varphi(D^{ss}_k)=0$ for $j=1,...,p$ and $k=1,...,q$. Then $\varphi$ is a metric anchor map, since $\varphi(\g)$ generates $\C^N$ and $h(\varphi(D),\varphi(D'))=x\varphi(D)^{\dagger}\varphi(D')$ is hermitian for $D,D'\in\g$ regardless of the specific value of $x\in\R\setminus\{0\}$. Let $\nabla:\g\times\C^N\rightarrow\C^N$
be the affine connection given by
\begin{equation*}
    \nabla_{D^{ab}_j} v=v(-D^{ab}_j+i\lambda^{ab}_j\One_N),\quad \nabla_{D^{ss}_k} v=v(-D^{ss}_k+i\lambda^{ss}_k\One_N).
\end{equation*}
From this it is clear that $\nabla_{D}\hat{v}_0=0$ for any $D\in\g$, implying that
\begin{equation*}
    \big((\Mat{N},\g_{\pi},\C^N,\varphi),h,\nabla\big)
\end{equation*} is indeed a real connection calculus by Lemma \ref{lem:RCC.cond.matrix}. Moreover, $\nabla$ is compatible with the metric, since each eigenvalue $i\lambda^{ab}_j$ and $i\lambda^{ss}_k$ is purely imaginary. Left is to consider the torsion:
\begin{align*}
    T_{\varphi}(D^{ab}_i,D^{ab}_j)&=\nabla_{D^{ab}_i}\varphi(D^{ab}_j)-\nabla_{D^{ab}_i}\varphi(D^{ab}_j)-\varphi([D^{ab}_i,D^{ab}_j])=0-0-\varphi(0)=0,\\
    T_{\varphi}(D^{ab}_j,D^{ss}_k)&=\nabla_{D^{ab}_j}\varphi(D^{ss}_k)-\nabla_{D^{ss}_k}\varphi(D^{ab}_j)-\varphi([D^{ab}_j,D^{ss}_k])=0-0-\varphi(0)=0,\\
    T_{\varphi}(D^{ss}_k,D^{ss}_l)&=\nabla_{D^{ss}_k}\varphi(D^{ss}_l)-\nabla_{D^{ss}_l}\varphi(D^{ss}_k)-\varphi([D^{ss}_k,D^{ss}_l])=0-0-0=0,
\end{align*}
where the last equality follows from $[D^{ss}_k,D^{ss}_l]\in\g^{ss}$ implying that $\varphi([D^{ss}_k,D^{ss}_l])=0$. Hence, $\nabla$ is the Levi-Civita connection and $\big((\Mat{N},\g_{\pi},\C^N,\varphi),h\big)$ is pseudo-Riemannian.

To prove necessity of the given condition, assume that $\big((\Mat{N},\g_{\pi},\C^N,\varphi),h\big)$ is pseudo-Riemannian and that $\nabla:\g\times\C^N\rightarrow\C^N$ is the Levi-Civita connection. Let $D_1,...,D_n$ be a basis of $\g$. Since $\varphi:\g\rightarrow\C^N$ is a metric anchor map there are $\mu_j\in\R$, $j=1,...,n$ such that $\varphi(D_j)=\mu_j \hat{v}_0$ for some for some $0\neq\hat{v}_0\in\C^N$. Since 
\begin{equation*}
    \big((\Mat{N},\g_{\pi},\C^N,\varphi),h,\nabla\big)
\end{equation*}
is a real connection calculus, and since there exist $\lambda_1,...,\lambda_n\in\R$ such that
\begin{equation*}
    \nabla_j v=v(-D_j+i\lambda_j\One)=i\lambda_j v-vD_j,
\end{equation*}
Lemma \ref{lem:RCC.cond.matrix} immediately implies that $\hat{v}_0$ is an eigenvector to $D_j$ (with corresponding eigenvalue $i\lambda_j$), $j=1,...,n$. Moreover, since the torsion vanishes, Proposition \ref{prop:torsion.semisimple.lie.alg} implies that $\g$ is not semisimple. This finishes the proof.
\end{proof}

\begin{corollary}
Let $\big((\Mat{N},\g_{\pi},\C^N),h\big)$ be a metric pre-calculus. If $\g$ is solvable,
then there exists a metric anchor map $\varphi:\g\rightarrow\C^N$ such that the resulting real metric calculus $\big((\Mat{N},\g_{\pi},\C^N,\varphi),h\big)$ is pseudo-Riemannian.
\end{corollary}
\begin{proof}
    If $\g$ is solvable, then by Lie's theorem the matrices in $\g$ have a common eigenvector $\hat{v}$. And since solvable Lie algebras are not semisimple, \ref{thm:riemannian.criterion.matrix} implies that there exists a metric anchor map $\varphi:\g\rightarrow\C^N$ such that $\big((\Mat{N},\g_{\pi},\C^N,\varphi),h\big)$ is pseudo-Riemannian.
\end{proof}

Below we give a simple example highlighting that both of the conditions in Theorem \ref{thm:riemannian.criterion.matrix} need to be checked separately when investigating whether a metric pre-calculus $\big((\Mat{N},\g_{\pi},\C^N),h\big)$ can be made into a pseudo-Riemannian calculus, as there exist semisimple matrix Lie algebras whose elements share a common eigenvector, as well as matrix Lie algebras that are not semisimple and whose elements do not share a common eigenvector.

\begin{example}\label{ex:matrix.case.illustration}
Let $(B_{\A},h)=\big((\Mat{N},\g_{\pi},\C^N),h\big)$ be a metric pre-calculus, let $\sigma_1,\sigma_2,\sigma_2\in\mathfrak{su}(2)$ be given by 
\begin{equation*}
    \sigma_1=\begin{pmatrix}
    0 & i\\
    i & 0
    \end{pmatrix},\quad \sigma_2=\begin{pmatrix}
    0 & 1\\
    -1 & 0
    \end{pmatrix},\quad \sigma_3=\begin{pmatrix}
    i & 0\\
    0 & -i
    \end{pmatrix},
\end{equation*}
and let $\One_2$ denote the identity matrix in $\Mat{2}$. We know that $\R\gen{\sigma_1,\sigma_2,\sigma_3}=\mathfrak{su}(2)$ is semisimple, and it is straightforward to check that there is no common eigenvector for $\sigma_1$, $\sigma_2$ and $\sigma_3$. These matrices enable us to construct a semisimple matrix Lie algebra whose elements have a common eigenvector, as well as a matrix Lie algebra that is not semisimple, but whose elements do not share a common eigenvector. We make the following choices:
\begin{align*}
    D_0&=\begin{pmatrix} i\cdot\One_2 & 0 \\ 0 & -i\cdot\One_2\end{pmatrix}\\
    D_j&=\begin{pmatrix} \sigma_j & 0 \\ 0 & \sigma_j\end{pmatrix},\quad j=1,2,3,\\
    D'_j&=\begin{pmatrix} 0 & 0 \\ 0 & \sigma_{j}\end{pmatrix},\quad j=1,2,3,
\end{align*}
and create three distinct Lie subalgebras of $\mathfrak{su}(4)$:
\begin{align*}
    \g^a&=\R\gen{D'_1,D'_2,D'_3}\simeq\mathfrak{su}(2),\\
    \g^b&=\R\gen{D_0,D_1,D_2,D_3}\simeq \mathfrak{ab}(1)\oplus\mathfrak{su}(2),\\
    \g^c&=\R\gen{D_0,D'_1,D'_2,D'_3}\simeq \mathfrak{ab}(1)\oplus\mathfrak{su}(2)\simeq \g^b.
\end{align*}
The necessary and sufficient conditions listed in Theorem \ref{thm:riemannian.criterion.matrix} on the matrix Lie algebra in question are that it is not semisimple and that all matrices have a common eigenvector. It is straightforward to check that $\g^b\simeq \g^c$ are not semisimple and that all matrices in $\g^a$ and in $\g^c$ share a common eigenvector $\hat{v}_0=(1,0,0,0)\in\C^4$. It is equally straightforward to check that $\g^a$ is semisimple and that there is no common eigenvector to all matrices in $\g^b$. Hence, Theorem \ref{thm:riemannian.criterion.matrix} tells us that there exists a metric anchor map $\varphi^c:\g^c\rightarrow\C^4$ such that the real metric calculus $\big((\Mat{4},\g^c_{\pi},\C^4,\varphi^c),h\big)$ is pseudo-Riemannian, and that no real metric calculus of the form $\big((\Mat{4},\g^a_{\pi},\C^4,\varphi),h\big)$ or $\big((\Mat{4},\g^b_{\pi},\C^4,\varphi^c),h\big)$ is pseudo-Riemannian.

More explicitly, a pseudo-Riemannian calculus $\big((\Mat{4},\g^c_{\pi},\C^4,\varphi^c),h\big)$ is given by the anchor map $\varphi^c(D_1)=\hat{v}_0=(1,0,0,0)$, $\varphi^c(D'_j)=0$ for $j=1,2,3$. Letting $\lambda_D$ denote the eigenvalue of $D\in\g^c$ with respect to $\hat{v}_0$ (i.e., $\hat{v}_0 D=\lambda_D\hat{v}_0$), the Levi-Civita connection $\nabla$ on $\C^N$ is given by
\begin{equation*}
    \nabla_{D} v=\lambda_D v- v D,\quad D\in\g^c.
\end{equation*}
\end{example}

\subsection{Generalizing the matrix case}
We now let $\A$ be an arbitrary unital $^*$-algebra such that $\C^N$ is a simple (right) $\A$-module. When considering metric pre-calculi $(B_{\A},h)=\big((\A,\g_{\pi},\C^N),h\big)$ where $\C^N$ is a simple and projective $\A$-module the discussion is similar to the case where $\A=\Mat{N}$, although certain care must be taken when considering the Lie algebra $\g$. We outline the details below.

For the sake of notational convenience in several of the proofs below, we let $\rho:\A\rightarrow\Mat{N}$ denote the $^*$-homomorphism defined by
\begin{equation*}
    v\rho(a)=v\cdot a,\quad a\in A, v\in\C^N.
\end{equation*}
Since $\C^N$ is a simple $\A$-module, it immediately follows from the Jacobson density theorem that $\rho$ is a surjection.

We are mainly concerned with the case where $\C^N$ is simple and projective and where $h$ is an invertible metric. This case is greatly simplified by the following general lemma.
\begin{lemma}[c.f. Proposition 2.6 in \cite{a:nc.min.surfaces}]\label{lem:inv.metric}
Let $M$ be a finitely generated projective (right) $\A$-module with generators $e_1,...,e_n$ and let $h:M\times M\rightarrow \A$ be a (invertible) metric. Setting $h_{ij}=h(e_i,e_j)$, there exist $h^{ij}\in\A$ such that $(h^{ij})^*=h^{ji}$ and $e_k h^{kl}h_{li}=e_i$ for $i,j=1,...,n$.
\end{lemma}
\noindent
The above lemma enables us to give the following characterization of the metric $h$.
\begin{proposition}\label{prop:proj.module}
    Let $\C^N$ be a simple $\A$-module, and let $h$ be an invertible metric on $\C^N$. Then $\C^N$ is a projective (right) $\A$-module if and only if there exists a $c\in\R\setminus\{0\}$ such that $\rho(h(u,v))=c u^{\dagger}v$ for all $u,v\in\C^N$.
\end{proposition}
\begin{proof}
    We begin proving necessity of the given condition. Since the composition $\rho\circ h:\C^N\times\C^N\rightarrow\Mat{N}$ must be a hermitian form, this immediately implies that there exists a $c\in\R$ such that $\rho(h(u,v))=c u^{\dagger}v$. To prove that $c\neq 0$, let $e_1\neq 0$ be a generator of $\C^N$ as an $\A$-module. Then, setting $h_{11}=h(e_1,e_1)$, Lemma \ref{lem:inv.metric} implies that there exist $h^{11}\in\A$ such that $e_1\cdot h^{11} h_{11}=e_1$. Since 
    \begin{equation*}
        0\neq e_1=e_1\cdot h^{11}h_{11}=e_1\rho(h^{11})\rho(h_{11}),
    \end{equation*}
    it follows that $\rho(h_{11})=c e_1^{\dagger}e_1\neq0$, which in turn implies that $c\neq 0$.
    
    Conversely, suppose that there is a $c\in\R\neq\{0\}$ such that $\rho(h(u,v))=c u^{\dagger}v$ for all $u,v\in\C^N$.
    Let $e_1=(1,0,...,0)\in\C^N$, and consider the (module) homomorphism $\iota:\C^N\rightarrow \A$ given by
    \begin{equation*}
        \iota(v)=h(e_1,v),\quad v\in\C^N.
    \end{equation*}
    Now, let $\kappa:\A\rightarrow \C^N$ be the (module) homomorphism given by $\kappa(a)=c^{-1}e_1\cdot a$.
    Then, for every $v\in\C^N$, we get
    \begin{equation*}
        \kappa\circ\iota(v)=\kappa(h(e_1,v))=c^{-1}e_1\cdot h(e_1,v)=c^{-1}e_1 c e_1^{\dagger}v=v,
    \end{equation*}
    implying that $\kappa\circ\iota=\operatorname{id}_{\C^N}$. Hence, the exact sequence
    \begin{equation*}
        0\rightarrow \ker\kappa\hookrightarrow \A\xrightarrow{\kappa}\C^N\rightarrow 0
    \end{equation*}
    splits and $\A=\ker\kappa\oplus\operatorname{im}\iota$, where $\operatorname{im}\iota\simeq\C^N$. In other words, $\C^N$ is a direct summand of the free $\A$-module $\A$ and is therefore projective.
\end{proof}
In light of Proposition \ref{prop:proj.module} we shall assume that $\C^N$ is a simple projective $\A$-module and that $h:\C^N\times\C^N\rightarrow\A$ is an invertible metric going forward. This ensures (among other things) the following useful result:
\begin{lemma}[cf. \cite{a:nc.min.surfaces}, Corollary 3.7]\label{lem:existence.of.metric.connection}
    Let $\C^N$ be a finitely generated projective $\A$-module and let $h$ be a metric on $\C^N$. If $\g$ is a real Lie algebra of hermitian derivations on $\A$, then there exists a connection $\nabla:\g\times\C^N\rightarrow\C^N$ that is compatible with $h$.
\end{lemma}
We use this fact to show that there is a unique representation $\bar{\rho}:\g\rightarrow\mathfrak{su}(N)$ that satisfies
\begin{equation*}
    \rho(\der(a))=[\bar{\rho}(\der),\rho(a)],\quad a\in\A,\quad\der\in\g,
\end{equation*}
which can be used to give a convenient parameterization of all affine connections that need to be considered.
\begin{proposition}\label{prop:lie.alg.rep.and.connection.parameterization}
    Let $\C^N$ be a simple and projective $\A$-module, and let $h$ be an invertible metric on $\C^N$. If $\g\subseteq\Der(\A)$ is a finite-dimensional real Lie algebra of hermitian derivations with basis $\der_1,...,\der_n$, then there exists a unique representation $\bar{\rho}:\g\rightarrow\mathfrak{su}(N)$ such that $\rho(\der(a))=[\bar{\rho}(\der),\rho(a)]$ for all $a\in\A$ and $\der\in\g$. Moreover, every affine connection $\nabla:\g\times\C^N\rightarrow\C^N$ is of the form
    \begin{equation*}
        \nabla_{\der_j} v=v(-\bar{\rho}(\der_j)+t_j\One_N),\quad v\in\C^N,
    \end{equation*}
    where $t_j\in\C$ for $j=1,...,n$. Furthermore, the connection $\nabla$ is compatible with $h$ if and only if $\operatorname{Re}(t_j)=0$ for $j=1,...,n$.
\end{proposition}
\begin{proof}
    Since $h$ is an invertible metric on the projective $\A$-module $\C^N$, Lemma \ref{lem:existence.of.metric.connection} implies that there exists at least one metric connection on $\g\times\C^N$. Hence,
    let $\nabla:\g\times \C^N\rightarrow \C^N$ denote an affine connection on $\C^N$. Since  $\nabla_j=\nabla_{\partial_j}$ is a linear map from $\C^N$ to itself, it follows that there are matrices $X_1,...X_n$ such that
\begin{equation*}
    \nabla_j v= vX_j;
\end{equation*}
by the Leibniz condition for $\nabla$, it follows that
\begin{equation*}
    v\rho(a)X_j=\nabla_j v\rho(a)=\nabla_j (v\cdot a)=(\nabla_j v)\cdot a+v\cdot\der_j(a)=vX_i\rho(a)+v\rho(\der_j(a))
\end{equation*}
for all $v\in\C^m$ and all $a\in\A$,
which is equivalent to 
\begin{equation}\label{eq:rho.of.derivation}
    \rho(\partial_j(a))=[\rho(a), X_j]=[-X_j,\rho(a)]
\end{equation}
for all $a\in\A$. This together with the Jacobi identity implies
\begin{align*}
    [\rho(a),[X_i,X_j]]&=-[X_i,[X_j,\rho(a)]]-[X_j,[\rho(a),X_i]]\\
    &=[X_i,\rho(\der_j(a))]-[X_j,\rho(\der_i(a))]=[\rho(\der_i(a)),X_j]-[\rho(\der_j(a)),X_i]\\
    &=\rho((\der_j\der_i-\der_i\der_j)(a))=\rho([\der_j,\der_i](a))=\rho(f^k_{ji}\der_k(a))=[\rho(a),f^k_{ji} X_k].
\end{align*}

The matrices $X_j$ enable us to construct a representation $\bar{\rho}:\g\rightarrow\mathfrak{su}(N)$ in the following way.
For each $j=1,...,n$, let $D_j=-X_j+t_j\One_N$, where $t_j=\frac{\tr(X_j)}{N}$ so that $\tr(D_j)=0$. Then, by Equation (\ref{eq:rho.of.derivation}), $\rho(\der_j(a))=[D_j,\rho(a)]$ for each $j$, and moreover we have that
\begin{equation*}
    [[D_i,D_j],\rho(a)]=[[X_i,X_j],\rho(a)]=[f^k_{ji}X_k,\rho(a)]=[f^k_{ij}D_k,\rho(a)],
\end{equation*}
showing that $[D_i,D_j]=f^k_{ij}D_k$, since $\rho$ is surjective and $\tr(D_k)=0$ for $k=1,...,n$.
Since $\der_i=\der_i^*$ (i.e., $(\der_i(a^*))^*=\der_i(a)$ for all $a\in\A$), it follows that
\begin{align*}
    [D_i,\rho(a)]&=\rho(\der_i(a))=\rho(\der^*_i(a))=\rho(\der_i(a^*)^*)=[D_i,\rho(a^*)]^{\dagger}\\
    &=[D_i,\rho(a)^{\dagger}]^{\dagger}=-[D_i^{\dagger},\rho(a)]=[-D_i^{\dagger},\rho(a)]
\end{align*}
for all $a\in\A$, where $\dagger$ denotes the hermitian transpose of a matrix. This is equivalent to $D_i+D_i^{\dagger}$ commuting with every element in $\rho(\A)$.
Since $\rho$ is surjective, this implies that $D_i+D_i^{\dagger}=\alpha_i\One_N$ for some $\alpha_i\in\R$ for every $i=1,...,n$. However, since the traces of both $D_i$ and $D_i^{\dagger}$ vanish it follows that each $\alpha_i=0$, i.e., that $D_i\in\mathfrak{su}(N)$ for $i=1,...,n$. Hence, by the above argument it follows that $\bar{\rho}:\g\rightarrow\mathfrak{su}(N)$ is defined by $\bar{\rho}(\der_i)=D_i$, $i=1,...,n$.

To prove uniqueness of $\bar{\rho}$, let $\rho':\g\rightarrow\mathfrak{su}(N)$ be another representation such that $\rho(\der(a))=[\rho'(\der),\rho(a)]$ for all $a\in\A$ and $\der\in\g$. Then it follows that
\begin{equation*}
    [\rho'(\der),\rho(a)]=[\bar{\rho}(\der),\rho(a)]\Leftrightarrow [\rho'(\der)-\bar{\rho}(\der),\rho(a)]=0,
\end{equation*}
implying that $\rho'(\der)-\bar{\rho}(\der)$ is a multiple of the identity matrix. And since $tr(\rho'(\der))=tr(\bar{\rho}(\der))=0$, it follows that $\rho'(\der)=\bar{\rho}(\der)$ for $\der\in\g$. 

Now, using $\bar{\rho}$ we find that
\begin{equation*}
    \nabla_j v=v X_j=v (-D_j+t_j\One_N)=v(-\bar{\rho}(\der_j)+t_j\One_N),\quad v\in\C^N,
\end{equation*}
and it is straightforward to see that the above formula defines an affine connection $\nabla$ for every $t_j\in\C$.

Now, let $\nabla$ be an affine connection that is compatible with $h$. This implies that
\begin{align*}
    \rho(\der_j(h(u,v)))&=\rho(h(\nabla_j u,v)+h(u,\nabla_j v))\\&=\rho(h(u (-\bar{\rho}(\der_j)+t_j\One_N),v)+h(u,v (-\bar{\rho}(\der_j)+t_j\One_N)))\\
    &=\rho(h(u (-\bar{\rho}(\der_j)+t_j\One_N),v))+\rho(h(u,v (-\bar{\rho}(\der_j)+t_j\One_N)))\\
    &=c\left( (-\bar{\rho}(\der_j)+t_j\One_N)^{\dagger} u^{\dagger}v+u^{\dagger}v (-\bar{\rho}(\der_j)+t_j\One_N)\right)\\
    &=c\left(\bar{\rho}(\der_j)u^{\dagger}v-u^{\dagger}v\bar{\rho}(\der_j)\right)+c(\bar{t}_j+t_j)u^{\dagger}v\\
    &=[\bar{\rho}(\der_j),\rho(h(u,v))]+c(\bar{t}_j+t_j)u^{\dagger}v.
\end{align*}
On the other hand, since $\rho(\der(a))=[\bar{\rho}(\der),\rho(a)]$ for all $\der\in\g$ and all $a\in\A$,
\begin{align*}
    \rho(\der_j(h(u,v)))=[\bar{\rho}(\der_j),\rho(h(u,v))],
\end{align*}
and hence $c(\bar{t}_j+t_j)u^{\dagger}v=0$ for all $u,v\in\C^N$, which is true if and only if $\bar{t}_j=-t_j$ for $j=1,...,n$.

Finally, to prove that every connection on the form
\begin{equation*}
    \nabla_j v=v X_j=v (-D_j+t_j\One_N)=v(-\bar{\rho}(\der_j)+t_j\One_N),\quad v\in\C^N,
\end{equation*}
is metric if $\operatorname{Re}(t_j)=0$ for $j=1,...,n$, note that Lemma \ref{lem:existence.of.metric.connection} implies that there exists at least one choice $\hat{t}_1,...,\hat{t}_n\in\C$ such that the connection
\begin{equation*}
    \hat{\nabla}_j v=v(-\bar{\rho}(\der_j)+\hat{t}_j\One_N),\quad v\in\C^N,
\end{equation*}
is compatible with $h$. Moreover, by the above argument it follows that $\Re(\hat{t}_j)=0$ for $j=1,...,n$.
Since $\rho$ is surjective and $\rho(\der_j(a))=[\bar{\rho}(\der_j),\rho(a)]$ for all $a\in\A$, it follows that there are $d_j=-d_j^*\in\A$ such that $\rho(d_j)=\bar{\rho}(\der_j)$. Hence, 
\begin{align*}
    h(\hat{\nabla}_j u,v)+h(u,\hat{\nabla}_j v)&=h(u(-\bar{\rho}(\der_j)+\hat{t}_j\One),v)+h(u,v(-\bar{\rho}(\der_j)+\hat{t}_j\One))\\
    &=-d^*_j h(u,v)-h(u,v)d_j+(\hat{t}_j+\bar{\hat{t}}_j)h(u,v)\\
    &=d_j h(u,v)-h(u,v)d_j+0=[d_j,h(u,v)].
\end{align*}
Since $\hat{\nabla}$ is compatible with $h$, it follows that
\begin{equation*}
    \der_j(h(u,v))=h(\hat{\nabla}_j u,v)+h(u,\hat{\nabla}_j v)=[d_j,h(u,v)].
\end{equation*}
Let $\nabla:\g\times\C^N\rightarrow\C^N$ be a connection defined by 
\begin{equation*}
    \nabla_j v=v(-\bar{\rho}(\der_j)+t_j\One_N),\quad v\in\C^N,
\end{equation*}
such that
$\Re(t_j)=0$, $j=1,...,n$. As was done explicitly for $\hat{\nabla}$, it is straightforward to show by direct computation that
\begin{equation*}
    h(\nabla_j u,v)+h(u,\nabla_j v)=[d_j,h(u,v)],
\end{equation*}
and since this expression is equal to $\der_j(h(u,v))$
for all $u,v\in\C^N$ and $j\in\{1,...,n\}$ it follows that $\nabla$ is compatible with $h$ as well. The statement follows. 
\end{proof}

Given a real Lie algebra of hermitian derivations, we need to consider the possible metric anchor maps $\varphi:\g\rightarrow\C^N$. 
Since the metric $h:M\times M\rightarrow \A$ is such that
\begin{equation*}
    u\cdot h(v,w)=c u(v^{\dagger}w),\quad c\in\R\setminus\{0\},
\end{equation*}
and since $v^{\dagger}w=w^{\dagger}v=(v^{\dagger}w)^{\dagger}$ if and only if there are $\mu_v,\mu_w\in\R$ and $0\neq \hat{v}_0\in\C^N$ such that $v=\mu_v \hat{v}_0$ and $w=\mu_w \hat{v}_0$, it follows by an argument analogous to the matrix case that $\varphi:\g\rightarrow \C^N$ is a metric anchor map if and only if $\varphi(\der_i)=\mu_i \hat{v}_0$ for a nonzero $\hat{v}_0$ and $\mu_i\in\R$ that are not all zero (note that $h(\mu_i\hat{v}_0,\mu_j\hat{v}_0)=\mu_i\mu_j h(\hat{v}_0,\hat{v}_0)$ is hermitian for all choices of $\mu_i,\mu_j\in\R$, ensuring sufficiency of the given condition). 

\begin{lemma}\label{lem:RCC.cond}
Let $(\A,\g_{\pi},\C^N,h)$ be a metric pre-calculus such that $\C^N$ is a simple projective (right) $\A$-module, and let $\varphi:\der_i\mapsto \mu_i \hat{v}_0$ be a metric anchor map (i.e., all $\mu_i$ are real), yielding a real metric calculus $(C_{\A},h)$.
Given a metric connection $\nabla$, $(C_{\A},h,\nabla)$ is a real connection calculus iff $\nabla_{\der} \hat{v}_0=0$, $\der\in\g$.
\end{lemma}
\begin{proof}
The argument is completely analogous to the proof of Lemma \ref{lem:RCC.cond.matrix}.
\end{proof}

Before stating the main result of this section, we give a brief description of the torsion $T_{\varphi}$ of a metric connection $\nabla$ given a metric anchor map $\varphi:\der_j\mapsto\mu_j\hat{v}_0$. As in the matrix case we get that
\begin{align*}
    T_{\varphi}(\der_i,\der_j)&=\nabla_i \varphi(\der_j)-\nabla_j \varphi(\der_i)-\varphi([\der_i,\der_j])\\
    &=\hat{v}_0(\mu_j (-\bar{\rho}(\der_i)+t_i\One_N)-\mu_i (-\bar{\rho}(\der_j)+t_j\One_N)-\mu_k f^k_{ij}\One),
\end{align*}
where $[\der_i,\der_j]=\der_k f^k_{ij}$. Like in the matrix case, we have that the matrix
\begin{equation*}
    \mu_j (-\bar{\rho}(\der_i)+t_i\One_N)-\mu_i (-\bar{\rho}(\der_j)+t_j\One_N)
\end{equation*}
is antihermitian, implying that the torsion cannot vanish unless $\mu_k f^k_{ij}=0$. Hence, if $\g$ is semisimple, Lemma \ref{lem:g.is.semisimple} immediately implies that the torsion does not vanish, in analogy with the matrix case. 

However, the general case is not completely analogous to the matrix case since the Levi decomposition $\g=\g^r\oplus\g^{ss}$ can no longer be assumed to be a direct sum, and hence it is assumed to be a semidirect sum going forward. The consequence of this is that the generalization of Theorem \ref{thm:riemannian.criterion.matrix} given below requires a consideration of a larger set of necessary and sufficient conditions compared to the matrix case.
\begin{proposition}
Let $((\A,\g_{\pi},\C^N),h)$ be a metric pre-calculus where $\C^N$ is a simple projective (right) $\A$-module, and let $\der_1,...,\der_{n^r}$ and $\der'_1,...,\der'_{n^{ss}}$ be bases for $\g^r$ and $\g^{ss}$, respectively, where $\g=\g^r\oplus\g^{ss}$ is the Levi decomposition. Then there exists a metric anchor map $\varphi$ such that $((\A,\g_{\pi},\C^N,\varphi),h)$ is pseudo-Riemannian if and only if there exists a common eigenvector $\hat{v}_0\in\C^N$ to all $D=\bar{\rho}(\der)$ for $\der\in\g$ and if there exists a nontrivial solution $\mu_1,...,\mu_{n^r}\in\R$ to the linear system
of equations
\begin{equation*}
    \mu_k r^k_{ij}=0,\quad \mu_k s^k_{pq}=0, \qquad i,j,p\in\{1,...,n^r\},\quad q\in\{1,...,n^{ss}\},
\end{equation*}
 where $[\der_i,\der_j]=r^k_{ij}\der_k$ and $[\der_p,\der'_q]=s^k_{pq}\der_k$ and $k$ ranges from $1$ to $n^r$.
\end{proposition}
\begin{proof}
For sufficiency, assume that $\hat{v}_0$ is a common eigenvector to all $D=\bar{\rho}(\der)$, $\der\in\g$, and that $\mu_1,...,\mu_{n^r}\in\R$ is a nontrivial solution to the system $\mu_k r^k_{ij}=0$, $i,j\in\{1,...,n^r\}$ and $\mu_k s^k_{pq}=0$, $p\in\{1,...,n^r\}$, $q\in\{1,...,n^{ss}\}$. Letting $\lambda_{\der}$ denote the eigenvalue of $\bar{\rho}(\der)$ for $\der\in\g$, we define the connection $\nabla$ given by the formula
\begin{equation*}
    \nabla_{\der_j} v=v(-\bar{\rho}(\der_j)+\lambda_{\der_j}\One_N),\quad \nabla_{\der'_k} v=v(-\bar{\rho}(\der'_k)+\lambda_{\der'_k}\One_N),
\end{equation*}
for $j=1,...,n^r$ and $k=1,...,n^{ss}$. By Proposition \ref{prop:lie.alg.rep.and.connection.parameterization}, $\nabla$ is an affine connection that is compatible with $h$, and by construction it is clear that $\nabla_{\der}\hat{v}_0=0$ for all $\der\in\g$. Define the metric anchor map $\varphi:\g\rightarrow\C^N$ by $\varphi(\der_j)=\mu_j\hat{v}_0$ for $j=1,...,n^r$ and $\varphi(\der'_k)=0$ for $k=1,...,n^{ss}$. It is straightforward to check that $\big((\A,\g_{\pi},\C^N,\varphi),h,\nabla\big)$ is a real connection calculus and that $\nabla$ has vanishing torsion.

Next we prove necessity of the given conditions in the proposition statement. If $\varphi:\g\rightarrow\C^N$ is a metric anchor map and $\nabla$ is a Levi-Civita connection such that $\big((\A,\g_{\pi},\C^N,\varphi),h,\nabla\big)$ is a real connection calculus, then by Proposition \ref{prop:lie.alg.rep.and.connection.parameterization} there exist $t_1,...,t_{n^r},t'_1,...,t'_{n^{ss}}\in\R$ such that
\begin{equation*}
    \nabla_{\der_j} v=v(-\bar{\rho}(\der_j)+it_j\One_N),\quad \nabla_{\der'_k} v=v(-\bar{\rho}(\der'_k)+it'_k\One_N).
\end{equation*}
Moreover, there exist $\mu_j,\mu'_k\in\R$ and $0\neq\hat{v}_0\in\C^N$ such that $\varphi(\der_j)=\mu_j\hat{v}_0$, $\varphi(\der'_k)=\mu'_{k}\hat{v}_0$, $j=1,...,n^r$ and $k=1,...,n^{ss}$ such that 
\begin{equation*}
    \{\mu_1,...,\mu_{n^r}\}\cup\{\mu'_1,...,\mu'_{n^{ss}}\}\neq\{0\}.
\end{equation*}
By Lemma \ref{lem:RCC.cond} it follows that
\begin{equation*}
    0=\nabla_{\der_j}\hat{v}_0=-\hat{v}_0\bar{\rho}(\der_j)+it_j\hat{v}_0,\quad 0=\nabla_{\der'_k}\hat{v}_0=-\hat{v}_0\bar{\rho}(\der'_k)+it'_k\hat{v}_0,
\end{equation*}
implying that $\hat{v}_0$ is an eigenvector of all matrices of the form $\bar{\rho}(\der)$, $\der\in\g$.

Checking the torsion $T_{\varphi}$, Lemma \ref{lem:RCC.cond} implies that
\begin{align*}
    T_{\varphi}(\der,\der')&=\nabla_{\der}\varphi(\der')-\nabla_{\der'}\varphi(\der)-\varphi([\der,\der'])\\
    &=0-0-\varphi([\der,\der'])=-\varphi([\der,\der'])=0,\quad \der,\der'\in\g.
\end{align*}
If $\der,\der'\in\g^{ss}$, then since $\g^{ss}$ is semisimple the bracket $[\der,\der']$ could be any element of $\g^{ss}$, implying that
$\varphi(\der)=0$ for $\der\in\g^{ss}$. This implies that $\{\mu'_1,...,\mu'_{n^{ss}}\}=\{0\}$, which in turn implies that $\{\mu_1,...,\mu_{n^r}\}\neq\{0\}$. By considering the vanishing torsion, we get that
\begin{align*}
    0&=T_{\varphi}(\der_i,\der_j)=-\varphi([\der_i,\der_k])=-\hat{v}_0(\mu_k r^k_{ij}),\quad i,j\in\{1,...,n^r\},\\
    0&=T_{\varphi}(\der_p,\der'_q)=-\varphi([\der_p,\der'_q])=-\hat{v}_0(\mu_k s^k_{pq}),\quad p\in\{1,...,n^r\}, q\in\{1,...,n^{ss}\}.
\end{align*}
Hence, $\mu_1,...,\mu_{n^r}$ is a nontrivial solution to the system
\begin{equation*}
    \mu_k r^k_{ij}=0,\quad \mu_k s^k_{pq}=0, \qquad i,j,p\in\{1,...,n^r\},\quad q\in\{1,...,n^{ss}\},
\end{equation*}
 where $[\der_i,\der_j]=r^k_{ij}\der_k$ and $[\der_p,\der'_q]=s^k_{pq}\der_k$ and $k$ ranges from $1$ to $n^r$. This completes the proof.
\end{proof}

\section{ General conditions for the existence of a Levi-Civita connection}\label{sec:gen.LC.conds}
As was showcased in the last section, for a general metric pre-calculus $(B_{\A},h)$ it is not guaranteed that there exists a metric anchor map $\varphi$ such that the resulting real metric calculus $(C_{\A}, h)$ is pseudo-Riemannian. In this section we shall derive necessary and sufficient criteria for the existence of a Levi-Civita connection for a given real metric calculus $\big((\A,\g_{\pi},M,\varphi),h\big)$ in the case where the module $M$ is finitely generated and projective. 

In order to state the main result of this section we utilize the characterization of real calculi where the module $M$ is projective used in \cite{tn:rc.finite.noncomm.spaces} to consider $M$ as a projection of a free $\A$-module of suitable rank. We go over the details of this characterization below.
Given a real  calculus $C_{\A}=(\A,\g_{\pi}\,M,\varphi)$ where $\dim\g=n$ and $M$ is projective, together with an invertible metric $h$ on $M$, we wish to investigate whether $(C_{\A},h)$ is pseudo-Riemannian, assuming that it is a real metric calculus. To do this, we note that if $\der_1,...,\der_n$ is a basis of $\g$, then $\varphi(\der_1),...,\varphi(\der_n)$ generate $M$ as an $\A$-module. We conclude that $\varphi$ can be viewed as a choice of generators $e_1,...,e_n$ of $M$, corresponding to the specific assignment $\der_i\mapsto e_i=\varphi(\der_i)$. Hence we may set $h_{ij}=h(e_i,e_j)$ and use Lemma \ref{lem:inv.metric} to guarantee that there exist $h^{ij}\in\A$ such that $e_k h^{kl}h_{li}=e_i$. 

Choosing an arbitrary basis $\hat{e}_1,...,\hat{e}_n$ of $\A^n$, we may create the module homomorphism $\phi:\A^n\rightarrow M$, defined by the formula
\begin{equation*}
    \phi(\hat{e}_i a^i)=e_i a^i.
\end{equation*}
Since $M$ is projective and $\phi$ is surjective, there exists a module homomorphism $\nu:M\rightarrow \A^n$ such that $\phi\circ\nu=\operatorname{id}_M$. Defining $p=\nu\circ\phi$, it is a standard fact that $p^2=p$ and that $p(\A^n)\simeq M$. From this discussion it becomes apparent that once a metric anchor map $\varphi:\der_i\mapsto e_i$ is given then it is always possible to give a projection $p:\A^n\rightarrow\A^n$ and a basis $\hat{e}_1,...,\hat{e}_n$ of $\A^n$ such that $e_i$  can be identified with $p(\hat{e}_i)=\hat{e}_k p^k_i$ for $i=1,...,n$. We note, in particular, that this identification means that the elements $h^{kl}\in\A$ can be assumed to satisfy the relationship $h^{kl}=p^k_m h^{ml}$ for all $m,l\in\{1,...,n\}$. Because, if $e_k h^{kl}h_{li}=e_i$ for $i=1,...,n$ one may set $\tilde{h} ^{kl}=p^k_m h^{ml}$, and then the identity $e_l p^l_i=e_i$ implies that
\begin{equation*}
    e_k\tilde{h}^{kl}h_{li}=e_k p^k_m h^{ml}h_{li}=e_m h^{ml}h_{li}=e_i;
\end{equation*}
we shall make this implicit assumption going forward.
\begin{example}\label{ex:proj.calculus.as.proj.of.free1}
    As a concrete example of the above procedure, let $\A=\Mat{2}$ and let $\g=\R\gen{D_1,D_2,D_3}=\mathfrak{su}(2)$, where $D_1,D_2$ and $D_3$ are given as in Example \ref{ex:su2}, i.e., 
\begin{equation*}
    D_1=\begin{pmatrix}
    0 & i\\
    i & 0
    \end{pmatrix},\quad D_2=\begin{pmatrix}
    0 & 1\\
    -1 & 0
    \end{pmatrix},\quad D_3=\begin{pmatrix}
    i & 0\\
    0 & -i
    \end{pmatrix},
\end{equation*}
with structure constants $f^k_{ij}$, given by 
\begin{align*}
    &(f^1_{12},f^2_{12},f^3_{12})=(0,0,-2)\\
    &(f^1_{13},f^2_{13},f^3_{13})=(0,2,0)\\
    &(f^1_{23},f^2_{23},f^3_{23})=(-2,0,0).
\end{align*}
As before, we let $\pi:\g\rightarrow\operatorname{Der}(\A)$ be the identification $\pi(D_i)=[D_i,\cdot]$.
As an $\A$-module we pick $M=\A$, with the obvious right action given by multiplication on the right. $M$ is clearly a free (and hence also projective) module, and as a metric we pick $h:M\times M\rightarrow \A$ given by $h(X,Y)=X^{\dagger}Y$. We describe the anchor maps $\varphi:\g\rightarrow M$ such that $\big((\A,\g_{\pi},M,\varphi),h\big)$ is a real metric calculus below, before characterizing $M=\A$ as a projection of $\A^3$.

In general, if $\varphi(D_k)=X_k\in M$ is to define an anchor map there have to exist matrices $Y^1, Y^2, Y^3$ such that 
$X_k Y^k=\One$, since otherwise the matrices $X_1, X_2, X_3$ clearly do not generate $\A$ as an $\A$-module; the matrices $Y^k$ are generally not uniquely determined by $X_k$, $k=1,2,3$. Since we are mainly interested in metric anchor maps $\varphi$, this immediately implies that $h_{ij}=h(X_i,X_j)=X_i^{\dagger}X_j=X_j^{\dagger} X_i=h_{ji}$ for $i,j=1,2,3$. It is straightforward to find the components of the inverse metric to be $h^{ij}=Y^i (Y^j)^{\dagger}$, as this yields
\begin{equation*}
    X_k h^{kl}h_{li}=X_k Y^k(Y^l)^{\dagger}X_l^{\dagger}X_i=(X_k Y^k)(X_l Y^l)^{\dagger}X_i=\One\cdot\One\cdot X_i=X_i,
\end{equation*}
as desired. We now construct a projection $p:\A^3\rightarrow \A^3$ such that $M=\A\simeq p(\A^3)$.

Let $\hat{e}_1=(\One,0,0)$, $\hat{e}_2=(0,\One,0)$ and $\hat{e}_3=(0,0,\One)$, making out a basis of $\A^3$. Let $\rho:\A^3\rightarrow M$ denote the module homomorphism given by $\rho(\hat{e}_k A^k)=X_k A^k$, which is an epimorphism since the matrices $X_k$ generate $\A$ as an $\A$-module. Next, we let $\nu:M\rightarrow\A^3$ be given by
$\nu(X)=(Y^1,Y^2,Y^3)X$; it is straightforward to verify that $\rho\circ\nu=\operatorname{id}_M$ due to the identity $X_k Y^k=\One$, and hence we get that our desired projection $p$ is given by $p=\nu\circ\rho$. More explicitly we find that
\begin{equation*}
    p(\hat{e}_i)=\nu\circ\rho(\hat{e}_i)=\nu(X_i)=(Y^1,Y^2,Y^3)X_i=\hat{e}_k Y^k X_i=\hat{e}_k p^k_i,
\end{equation*}
implying that the projection coefficients are given by $p^k_i=Y^k X_i$. It is straightforward to verify that that $p^k_m h^{ml}=h^{kl}$.
\end{example}
Given the above discussion, the following result can be used to determine whether a real metric calculus $\big((\A,\g_{\pi},M,\varphi),h\big)$ is pseudo-Riemannian whenever $M$ is a finitely generated projective module.

\begin{proposition}\label{prop:LC.cond}
Let $\big((\A,\g_{\pi}, p(\A^n),\varphi),h\big)$ be a real metric calculus where $\dim \g=n$ and where $p:\A^n\rightarrow\A^n$ is a projection. Let $\der_1,...,\der_n$ be a basis of $\g$, and let $\hat{e}_1,...\hat{e}_n$ be a basis of $\A^n$ such that $\varphi(\der_i)=e_i=p(\hat{e}_i)$ for $i=1,...,n$. Writing $h_{ij}=h(e_i,e_j)$ and $p(\hat{e}_ia^i)=\hat{e}_k p^k_i a^i$, then $\big((\A,\g_{\pi},p(\A^n),\varphi),h\big)$ is a pseudo-Riemannian calculus if and only if
\begin{equation}\label{eqn:LC.cond}
    p^k_l\der_i(p^l_j)=\Lambda^k_{il}(\delta^l_j\One-p^l_j),\quad i,j,k=1,...,n,
\end{equation}
where
\begin{equation*}
    \Lambda^k_{ij}=\frac{1}{2}h^{kl}\left(\der_i(h_{jl})+\der_j(h_{il})-\der_l(h_{ij})-h_{jq}f_{il}^q-h_{iq}f_{jl}^q+h_{lq}f_{ij}^q\right)
\end{equation*}
and $[\der_i,\der_j]= f^k_{ij}\der_k$.
\end{proposition}

\begin{proof}
We begin by proving necessity of (\ref{eqn:LC.cond}) for there to exist a Levi-Civita connection.
Let $\tilde{\nabla}$ be an arbitrary connection on the free module $\A^n$ defined by
\begin{equation*}
    \tilde{\nabla}_i\hat{e}_j=\hat{e}_k\tilde{\Gamma}^k_{ij}.
\end{equation*}
Then $\nabla:=p\circ\tilde{\nabla}|_{p(\A^n)}$ is an affine connection on $\g\times p(\A^n)$ given by
\begin{equation*}
    \nabla_i e_j=p(\tilde{\nabla}_i \hat{e}_k p^k_j)=p(\hat{e}_l\tilde{\Gamma}^l_{ik}p^k_j+\hat{e}_k\der_i(p^k_j))=e_l(\tilde{\Gamma}^l_{ik}p^k_j+\der_i(p^l_j)).
\end{equation*}
Using the notation $\Lambda^k_{ij}$ as in the proposition statement 
to simplify expressions going forward, it is straightforward to verify that Koszul's formula (\ref{eqn:Koszul}) from Proposition \ref{prop:koszul.formula} implies that
\begin{align*}
    h(e_m,\nabla_i e_j)
    &=\frac{1}{2}\left(\der_i h_{jk}+\der_j h_{ik}-\der_k h_{ij}
    -h_{il}f^l_{jk}+h_{jl}f^l_{ki}+h_{kl}f^l_{ij}\right)=h(e_m,e_l \Lambda^l_{ij})
\end{align*}
which is equivalent to stating that
\begin{equation*}
    e_l(\tilde{\Gamma}^l_{ik}p^k_j+\der_i(p^l_j)-\Lambda^l_{ij})=\hat{e}_q p^q_l(\tilde{\Gamma}^l_{ik}p^k_j+\der_i(p^l_j)-\Lambda^l_{ij})=0.
\end{equation*}
In other words,
\begin{equation}\label{eqn:LC-conn.cond}
    p^q_l\tilde{\Gamma}^l_{ik}p^k_j+p^q_l\der_i(p^l_j)=p^q_l\Lambda^l_{ij}=\Lambda^q_{ij},
\end{equation}
since $h^{qr}=p^q_l h^{lr}$.
Multiplying this equation by $p^j_m$ from the right, we find that
\begin{equation*}
    p^q_l\tilde{\Gamma}^l_{ik}p^k_m=\Lambda^q_{ij} p^j_m,
\end{equation*}
since $p^q_l\der_i(p^l_j)p^j_m=0$ for any projection $p$.
If we use this in equation (\ref{eqn:LC-conn.cond}) we see that
\begin{equation*}
    \Lambda^q_{ik}p^k_j+p^q_l\der_i(p^l_j)=\Lambda^q_{ij},
\end{equation*}
which is equivalent to
\begin{equation*}
    p^q_l\der_i(p^l_j)=\Lambda^q_{ik}(\delta^k_j\One-p^k_j),
\end{equation*}
i.e., Equation (\ref{eqn:LC.cond}). 

Proving sufficiency of the above condition, let
\begin{equation*}
    \tilde{\Gamma}^q_{im}=\Lambda^q_{ij}p^j_m,
\end{equation*}
and define the connection $\nabla:\g\times p(\A^n)\rightarrow p(\A^n)$ by
\begin{equation*}
    \nabla_i e_j=p(\hat{e}_l\tilde{\Gamma}^l_{ik}p^k_j+\hat{e}_k\der_i(p^k_j))=e_l(\tilde{\Gamma}^l_{ik}p^k_j+\der_i(p^l_j)).
\end{equation*}
Then
\begin{align*}
    h(e_m,\nabla_{i}e_j)&=h(e_m,e_l(\tilde{\Gamma}^l_{ik}p^k_j+\der_i(p^l_j)))\\
    &=h(e_m,\hat{e}_q p^q_r p^r_l(\tilde{\Gamma}^l_{ik}p^k_j+\der_i(p^l_j)))=h(e_m,e_r p^r_l(\Lambda^l_{ik}p^k_j+\der_i(p^l_j)))\\
    &=h(e_m,e_r (\Lambda^r_{ik}p^k_j+p^r_l\der_i(p^l_j)))=h(e_m,e_r (\Lambda^r_{ik}p^k_j+\Lambda^r_{ik}(\delta^k_j\One-p^k_j)))\\
    &=h(e_m,e_r\Lambda^r_{ik}),
\end{align*}
and hence $\nabla$ satisfies Koszul's formula, implying that $\big((\A,\g_{\pi},M,\varphi),h\big)$ is pseudo-Riemannian by Proposition \ref{prop:koszul.formula}. The statement follows.
\end{proof}

\begin{example}\label{ex:free.real.calculus}
Besides for checking whether a given real metric calculus over a projective module is pseudo-Riemannian, Proposition \ref{prop:LC.cond} can be useful when considering anchor maps that share specific properties. As an example of this, we consider the metric pre-calculus $(B_{\A},h)=\big((\A,\g_{\pi},\A^n),h\big)$ where $\dim\g=n$. Given a basis $\der_1,...,\der_n$ of $\g$, let $\varphi:\g\rightarrow \A^n$ be a metric anchor map such that $\varphi(\der_1),...,\varphi(\der_n)$ is a basis of $\A^n$. $\A^n$ can be considered as a projection of itself under the trivial projection $p=\operatorname{id}_{\A^n}$, with projection coefficients $p^k_l=\delta^k_l\One$ in the basis $\varphi(\der_1),...,\varphi(\der_n)$. Inserting these into Equation (\ref{eqn:LC.cond}) of Proposition \ref{prop:LC.cond} one finds that equality holds, since
\begin{equation*}
    p^k_l\der_i(p^l_j)=\delta^k_l\der_i(\delta^l_j\One)=0
\end{equation*}
and 
\begin{equation*}
    \Lambda^k_{il}(\delta^l_j\One-p^l_j)=\Lambda^k_{il}(\delta^l_j\One-\delta^l_j\One)=0
\end{equation*}
are trivially true for all $i,j,k\in\{1,...,n\}$. Hence $\big((\A,\g_{\pi},\A^n,\varphi),h\big)$ is pseudo-Riemannian whenever $\varphi$ is a metric anchor map constructed as above. 

Structures of the form $\big((\A,\g_{\pi},\A^n,\varphi),h\big)$ were studied in detail in \cite{atn:nc.minimal.embeddings} and \cite{tn:rc.finite.noncomm.spaces}, where they are referred to as \emph{free real metric calculi}, and the above argument using Proposition \ref{prop:LC.cond} is an alternative proof of Proposition 5.3 in \cite{atn:nc.minimal.embeddings}, which more clearly highlights that the result is a consequence of a general fact about real calculi where the module is projective.
\end{example}

\begin{example}\label{ex:matrix.free.rank1.nonexistent.LC}
We now expand on Example \ref{ex:proj.calculus.as.proj.of.free1}, where $\A=M=\Mat{2}$, and $\g=\mathfrak{su}(2)$ 
with structure constants $f^k_{ij}$, given by 
\begin{align*}
    &(f^1_{12},f^2_{12},f^3_{12})=(0,0,-2)\\
    &(f^1_{13},f^2_{13},f^3_{13})=(0,2,0)\\
    &(f^1_{23},f^2_{23},f^3_{23})=(-2,0,0).
\end{align*}
As a metric we had $h:M\times M\rightarrow \A$ given by $h(X,Y)=X^{\dagger}Y$. In Example \ref{ex:proj.calculus.as.proj.of.free1} we considered general anchor maps $\varphi:D_i\mapsto X_i$, and here we make the specific choice $X_1=\One_2$, $X_2=X_3=0$ and $Y^1=\One_2$, $Y^2=Y^3=0$. As outlined earlier, this implies that the coefficient $p^k_i=Y^k X_i$ is nonzero if and only if $k=i=1$, and $p^1_1=\One$. Moreover, setting $h_{ij}=h(\varphi(D_i),\varphi(D_j))$ for $i,j\in\{1,2,3\}$, it is straightforward to verify that $h_{11}=\One_2$ and that $h_{ij}=0$ for all other choices of $i$ and $j$. Plugging the projection coefficients into Equation (\ref{eqn:LC.cond}), the left-hand side becomes $p^k_l\der_i(p^l_j)=0$ for all $i,j,k$. Moreover, we get that the right-hand side of Equation (\ref{eqn:LC.cond}) reduces to
\begin{equation*}
    \Lambda^k_{ij}-\Lambda^k_{il}p^l_j=\Lambda^k_{ij}(\One-Y^1 X_j)=\Lambda^k_{ij}
\end{equation*}
if $j\neq 1$, and zero otherwise.
Calculating $\Lambda^1_{23}$ explicitly, we find that 
\begin{align*}
    \Lambda^1_{23}&=\frac{1}{2}h^{1l}\left(\der_2(h_{3l})+\der_3(h_{2l})-\der_l(h_{23})-h_{3q}f_{2l}^q-h_{2q}f_{3l}^q+h_{lq}f_{23}^q\right)\\
    &=\frac{1}{2}\One\left(0+0-0-0-0+\One f_{23}^1\right)=-\One\neq 0.
\end{align*}
Hence, with the specific choice of metric anchor map $\varphi$ given by $X_1=\One$ and $X_2=X_3=0$, we find that the resulting real metric calculus $\big((\A,\g_{\pi},\A,\varphi),h\big)$ is not pseudo-Riemannian.
\end{example}

The complications that arise in the above example, where it is difficult to determine whether there exists a metric anchor map that makes the given metric pre-calculus pseudo-Riemannian, are ultimately related to the structure of the Lie algebra $\g=\mathfrak{su}(2)$. Some of these problems vanish in cases where $\g$ is abelian, in part due to the fact that the expression for $\Lambda^k_{ij}$ is simplified by the fact that terms of the form $f^{k}_{ij}h_{jl}$ vanish when the structure constants are zero.

\begin{proposition}\label{prop:LC.free.mod.abelian}
Let $\A$ be a unital $^*$-algebra, $\g\subseteq\operatorname{Der}(\A)$ be an abelian Lie algebra of hermitian derivations with basis $\{\der_i\}_{i=1}^n$ and $M=\A^m$ for $m \leq \dim \g$. If $h:M\times M\rightarrow\A$ is an invertible metric such that there is a basis $\hat{e}_1,...\hat{e}_m$ of $\A^m$ satisfying
\begin{equation*}
    h(\hat{e}_i,\hat{e}_j)=h(\hat{e}_i,\hat{e}_j)^*\text{ and } \der(h(\hat{e}_i,\hat{e}_j))=0,\quad i,j\in\{1,...,m\},\quad \der\in\g,
\end{equation*}
then the metric anchor map $\varphi:\g\rightarrow M$ given by $\varphi(\der_i)=\hat{e}_i$ for $i=1,...,m$ and $\varphi(\der_j)=0$ for $j=m+1,...,n$ is such that
$\big((\A,\g_{\pi}\,M,\varphi),h\big)$ is pseudo-Riemannian.
\end{proposition}
\begin{proof}
Let $n=\dim\g$.
We begin by extending $\hat{e}_1,...,\hat{e}_m$ to a basis $\{\tilde{e}_k\}_1^n$ of $\A^n$ so that $\iota(\hat{e}_i)=\tilde{e}_i$ for $i=1,...,m$, where $\iota:\A^m\rightarrow\A^n$ denotes the inclusion map. Next, let $p:\A^n\rightarrow\A^n$ be the projection such that $p(\tilde{e}_i)=\tilde{e}_i$ for $i=1,...,m$ and $p(\tilde{e}_j)=0$ for $j=m+1,...,n$. By construction, $p(\A^n)\simeq \A^m$ under the obvious isomorphism $\tilde{e}_i\mapsto\hat{e}_i$ for $i=1,...,m$, and if we write $p(\tilde{e}_i)=\tilde{e}_k p^k_i$, then $p^k_i=\delta^k_i\One$ if $i\leq m$ and $p^k_j=0$ if $j>m$ for $k=1,...,n$.

Letting $\varphi:\g\rightarrow\A^m$ be the given metric anchor map, it follows that
\begin{equation*}
    \der(h_{ij})=\der(h(\varphi(\der_i),\varphi(\der_j)))=\der(h(\hat{e}_i,\hat{e}_j))=0
\end{equation*} 
for all $\der\in\g$. Moreover, since $p^k_i$ is also either $0$ or $\One$, it follows that for all $i,j,k$,
\begin{align*}
     p^k_l\der_i(p^l_j)&=0,\\
     \Lambda^k_{ij}&=\frac{1}{2}h^{kl}\left(\der_i(h_{jl})+\der_j(h_{il})-\der_l(h_{ij})-h_{jq}f_{il}^q-h_{iq}f_{jl}^q+h_{lq}f_{ij}^q\right)=0,
\end{align*}
since $\g$ being abelian implies that all structure constants $f^k_{ij}=0$. Hence, Proposition \ref{prop:LC.cond} can be directly applied to verify that $\big((\A,\g_{\pi}\,M,\varphi),h\big)$ is indeed pseudo-Riemannian, as desired.
\end{proof}

\section{Summary}\label{sec:summary}
The over-arching goal of this article has been to investigate the existence of Levi-Civita connections in the context of real calculi over projective modules and what effects the choice of anchor map may have in this regard. Given a metric pre-calculus, the general problem of determining whether there is a metric anchor map making the pre-calculus into a pseudo-Riemannian calculus is non-trivial. As an important step in the overall process of achieving a full understanding of the problem, the current article confirms that there are cases where no Levi-Civita connection exists for any metric anchor map (see Example \ref{ex:su2}), even if one only considers projective modules.

More broadly, the results in Section \ref{sec:fin.dim.modules} (as well as Example \ref{ex:matrix.free.rank1.nonexistent.LC}) indicate that $\g$ being semisimple acts as an obstruction to the existence of a Levi-Civita connection in several cases. In fact, in the case where $\g$ is semisimple and of dimension $n$, the only scenario where existence of a metric anchor map such that the resulting structure is a pseudo-Riemannian calculus has been verified is if $M\simeq \A^n$, as was shown in Example \ref{ex:free.real.calculus}; at the time of writing, if $\dim \g=n$ and $\g$ is semisimple we are not aware of any other example of a metric pre-calculus $\big((\A,\g_{\pi},M),h\big)$ where $M$ is projective for which there exists a metric anchor map $\varphi:\g\rightarrow M$ such that $\big((\A,\g_{\pi},M,\varphi),h\big)$ is pseudo-Riemannian.

In contrast to the case where $\g$ is semisimple, in cases where $\g$ is abelian all examples considered of metric pre-calculi $\big((\A,\g_{\pi},M),h\big)$ where $M$ is projective are such that there exists a metric anchor map $\varphi$ such that $\big((\A,\g_{\pi},M,\varphi),h\big)$ is pseudo-Riemannian (for instance, consider \cite{aw:curvature.three.sphere}, \cite{al:noncommutative.cylinder}, \cite{atn:nc.minimal.embeddings}). For the moment it is not clear whether there exists any counterexamples when $\g$ is abelian.

In conclusion, several important steps have been taken in understanding of the question of existence of the Levi-Civita connection when given geometric data in the form of a metric pre-calculus. In particular it is worth noting the substantial progress that has been achieved in the case when $\A=\Mat{N}$, and
in future work we hope to give a full characterization of metric pre-calculi for which a Levi-Civita connection exists.

\bibliographystyle{alpha} 
\bibliography{references} 

\end{document}